\documentclass[twoside,12pt]{article}
\usepackage{amsmath}
\usepackage{amsfonts}\usepackage{amsthm}
\usepackage{graphicx,pgf}
\usepackage{tikz}
\newtheorem{thm}{Theorem}[section]
\newtheorem{cor}[thm]{Corollary}

\newtheorem{lemma}[thm]{Lemma}
\newtheorem{remark}{Remark}[section]

\newcommand{\vectwo}[2]{ \left( \begin{array}{c} #1 \\ #2 \end{array} \right)}

\newcommand{\mattwo}[4]{\left( \begin{array}{cc} #1 & #2 \\ #3 & #4 \end{array}
\right)}
\renewcommand {\theequation}{\thesection.\arabic{equation}}

\newcommand{\ds}{\displaystyle}
\newcommand{\beq}{\begin{equation}}
\newcommand{\eeq}{\end{equation}}

\begin{document}
\title{\bf Front propagation in both directions and
coexistence of 
traveling fronts and pulses} 

\author{Chao-Nien Chen \thanks{Department of Mathematics, National Tsinghua University, Taiwan, ROC ({\tt chen@math.nthu.edu.tw}).} 
\and Y.S. Choi \thanks{Department of Mathematics, University of Connecticut, Storrs, CT 06269-3009  ({\tt choi@math.uconn.edu)}.}}

\date{}
\maketitle

\begin{quote}
\textbf{Abstract}: 
 For scalar reaction-diffusion {equations},
a traveling wave is a front which transforms a higher energy state to a lower energy state.  
The same is true for a system of equations with a gradient structure.
At the core of this phenomenon, the wave propagation results in an invasion of an equilibrium state from a different one. This paper is aimed at two questions: \\
\noindent
(i) Can reaction-diffusion waves exhibit front propagation in both directions between 
 two  distinct equilibrium states? \\ 
\noindent
(ii) The co-existence of traveling fronts and pulses; a subject rarely studied. \\
\noindent
Working on a FitzHugh-Nagumo model with the same physical parameters, 
we give positive answers to both questions.


\end{quote}

\section{Introduction} \label{sec_intro}
\setcounter{equation}{0}
\renewcommand{\theequation}{\thesection.\arabic{equation}}

The phenomena exhibited by traveling {waves} provide very useful information for 
{understanding} the
 dynamics of evolution. 
 Front propagation is 
 ubiquitous in diverse fields such as 
 phase transitions \cite{AC,AFM}, combustion \cite{BLL} and population dynamics \cite{AW}. 
Various aspects of front propagation have been investigated in 
extensive 
 literature 
(see, e.g. \cite{AW,BHN,BHM,BLL,BN,CCH,F,He,HPS,KR,LMN,LMN1,Mu,RTV,V,VVV} and references
therein; 
 of course the list is far from complete). \\ 


 
 The mathematical analysis for {scalar} reaction-diffusion equations
\begin{eqnarray}\label{seq}
u_t=\Delta u+g(u)
\end{eqnarray}
{has} been carefully deliberated;
well-known examples 
include Fisher-KPP nonlinearity $g(u)=u(1-u)$,
and  Allen-Cahn bistable nonlinearity $g(u)=u(1-u)(u-\beta)$, $\beta \in (0,1)$. 
{Constant equilibrium {states} of (\ref{seq}) satisfy $g(u)=0$.
A} traveling front is a steadily moving profile that connects
distinct equilibrium {states} at different ends.
 At the core of this phenomenon, the wave 
  propagation results in 
an invasion of an equilibrium state {by} a different one. \\
 
{A variational formulation, involving} the Lagrangian 
  $L(u,\nabla u)=\frac{1}{2}\vert\nabla u\vert^2 + G(u)$ {with a potential $G(u)= - \int_0^u g(\xi) \,d\xi$,} can be employed to {study the steady states}
   and 
the  traveling waves 
 of (\ref{seq}). 
 A planar traveling front 
 with positive velocity 
 $c$ is a function of the form $u(x-ct)$ which satisfies 
\begin{equation} \label{dissp}  
L(u(\infty),0)=G(u(\infty))> G(u(-\infty))=L(u(-\infty),0);\; 
 \end{equation}
as $t$ becomes large, the invader (the equilibrium state at $-\infty$), {which 
{at a} lower energy level,} wins. {This intuition 
 persists for 
 traveling waves on a cylinder \cite{BN,He,LMN1,V}, 
 even if a non-constant steady state sits on 
  the cross section at one end of the wave front. The same principle holds  
{for} traveling wave solutions in a reaction-diffusion system with gradient structure \cite{Mu}.
  By the law of dissipation in energy, given two distinct {steady} 
 states {in such systems},  
 one state can {never} act as an invader {and at the same time,} being deposed with respect to the
other, even starting from different initial data. 
For {other general} reaction-diffusion systems, it is {therefore} a common belief that 
 the invader has a 
lower energy than the deposed one. 
We examine this plausible myth, {which is} partly 
 motivated from the investigation of pattern formation \cite{AC,AFM,AW,BLL,CC,CH1,CKM,CT,CR,DY2,GS,KS,L,NAY,Ni,N2,O,PY,RS,RW,WW}. \\

Here is a question which seems to be rarely studied in mathematical literature: {\em Can reaction-diffusion waves exhibit front propagation in both directions between 
 two  distinct equilibrium states? } \\
One of the aims in this paper is to give a positive answer for the above questions by investigating 
the 
 traveling waves (or dissipative solitons \cite{L}) 
 of the FitzHugh-Nagumo equations
\begin{eqnarray} \label{FN}
\left\{  \begin{array}{rl}
\displaystyle {u}_t& \displaystyle =u_{xx}+\frac{1}{d}(f(u)-v) , \\ \\
\displaystyle v_t&= v_{xx}+u-\gamma v .
\end{array} \right.
\end{eqnarray}
Here $d>0$, $\gamma>0$ and $f(u) \equiv u(u-\beta)(1-u)$ with $0<\beta<1/2$. There are certain properties for the {nullclines} 
$v=u/\gamma$ and $v=f(u)$ {when $\gamma$ takes up some special values.}
\begin{enumerate}
\item [(N1)]
 If $\gamma > 4/(1-\beta)^2$, the {nullclines} 
intersect at three points; namely 
$(0,0)$, $(\mu_2, \mu_2/\gamma)$  
and $(\mu_3,\mu_3/\gamma)$, which represent the constant {equilibria} of (\ref{FN}). 
 Note that 
$f(\mu_i)=\mu_i/\gamma$ for $i=2,3$, $\mu_2+\mu_3=1+\beta$ and $0<\beta<\mu_2<(1+\beta)/2<\mu_3<1$. 

\item [(N2)]
Let $\hat{\rho} \equiv
\{1+\beta+\sqrt{\beta^2-\beta+1} \}/3$, the unique point at which the function 
$f$ attains a local maximum. 
Clearly there is a 
$\tilde{\gamma_1}>4/(1-\beta)^2$ such that $f(\hat{\rho})=\hat{\rho}/\tilde{\gamma_1}$. 
If $\gamma > \tilde{\gamma_1}$, then $f'(\mu_2)>0$ and $f'(\mu_3)<0$. 




\item [(N3)]
Let 
$\gamma_* \equiv \frac{9}{(1-2\beta) (2 -\beta)}$. 
When $\gamma=\gamma_*$ we denote the intersection points of two {nullclines} 
 by $(\mu_3^*,\frac{\mu_3^*}{\gamma_*})$, $({\mu_2^*},\frac{\mu_2^*}{\gamma_*})$ and $(0,0)$. Observe that two regions enclosed by the line $v=u/\gamma_*$ and the curve $v=f(u)$ are equal 
in area 
with opposing signs. Since both {nullclines} are antisymmetric with respect to the point $({\mu_2^*},\frac{\mu_2^*}{\gamma_*})$ 
  on the $(u,v)$ phase plane, 
  it {is} easily seen that 
 $\mu_3^*=\frac{2(1+\beta)}{3}$,
$\mu_2^*=\mu_3^*/2=(1+\beta)/3$ 
{and $f'(\mu_3^*)=f'(0)=-\beta$.}  
Also, the two regions enclosed by the horizontal line $v=f(\mu_2^*)$ and the curve $v=f(u)$ are equal in area with
opposing signs.
\end{enumerate}

 

The steady states of (\ref{FN}) satisfy
\begin{eqnarray} \label{FNS}
\left\{  \begin{array}{rl}
du_{xx}+f(u)-v = 0, \\ \\
v_{xx}+u-\gamma v = 0.
\end{array} \right.
\end{eqnarray}
{By virtue} of solving $v$ from the second equation and substituting 
 it into the {first one} yields an integral-differential equation 
\[
du_{xx}+f(u)- (\gamma -\Delta)^{-1}u=0. 
\]
{The associated Lagrangian  is}
   $L_{{\gamma}}(u,u_x) \equiv \frac{d}{2}u_x^2 + \frac{1}{2}u(\gamma-\Delta)^{-1}u +F(u)$. 
 The energy level of a constant steady state $(u,v)=(\mu, \mu/\gamma)$ can 
  be directly calculated as $L_{\gamma}(\mu,0) = \int_0^{\mu} (\xi/\gamma -f (\xi)) \, d\xi = \frac{\mu^2}{2 \gamma} + F(\mu)$.
{Using} the above information {for} 
 the nullclines,  
  if $\gamma > \gamma_*$ it follows from (N3) 
   that $L_{\gamma}(\mu_2,0)  > L_{\gamma}(0,0) > L_{\gamma}(\mu_3,0)$, {ranking} the order of the energy levels for 
  the three constant steady states $(\mu_2, \mu_2/\gamma)$, $(0,0)$ and 
$(\mu_3,\mu_3/\gamma)$ respectively. It has been shown \cite{CCH} that if $d > \gamma^{-2} $ 
 then (\ref{FN}) has a 
 traveling front solution 
$(c,u,v)$ 
such that $c > 0$, $\lim_{x \to \infty}(u,v) = (0,0)$ 
  and $\lim_{x \to -\infty}(u,v) = (\mu_3,\mu_3/\gamma)$. Indeed 
 taking the ansatz as in $\cite{He}$, $(c,u,v)$ satisfies 
\begin{eqnarray} \label{mainc}
\left\{  \begin{array}{rl}
\displaystyle dc^2 u_{xx}+dc^2 u_x+f(u)-v &= 0, \\ \\
\displaystyle c^2 v_{xx}+c^2 v_x+u-\gamma v &=0 .
\end{array} \right.
\end{eqnarray}

We employ variational argument to establish  
connecting orbits of {(\ref{mainc})}. Let $L^2_{ex}=L^2_{ex}({\bf R}) \equiv \{ u: \int_{-\infty}^{\infty} e^{x} (u(x))^2 \; dx < \infty \}$
be a Hilbert space equipped with a weighted norm
$\| u\|_{L^2_{ex}}\equiv \sqrt{ \int_{-\infty}^{\infty} e^{x}\, u^2 \; dx }$. 
For a given $u\in L^2_{ex}({\bf R})$, we {define} 
 \begin{equation} \label{v_Nsoln}
 v(x)= {\cal L}_c u \; (x)
 \equiv \int_{-\infty}^{\infty} G(x,s) \, u(s) \, ds,
 \end{equation}
where $G$ is a {Green's} function for the differential operator $(\gamma- c^2 \frac{d^2}{dx^2}-c^2 \frac{d}{dx})$. 
 Reversing the {locations of the two steady states} if necessary, we always consider {a} wave with 
  $c>0$. 
It is known \cite{CC1} that ${\cal L}_c: L^2_{ex} \to L^2_{ex}$ is self-adjoint with respect to the $L^2_{ex}$ inner product. 
Set $F(\xi)=-  \int_0^{\xi} f(\eta) \; d\eta
=\xi^4/4- (1+\beta)\xi^3/3+\beta \xi^2/2$.
Consider a functional
$J_c: {H^1_{ex}} \to {\bf R}$ defined by
\begin{equation} \label{def_J}
{J}_c(w) \equiv \int_{\bf R}e^{x} \{ \frac{dc^2}{2} w_x^2 + \frac{1}{2} w \, {\cal L}_c w + F(w) \} \; dx
\end{equation}
 {on the Hilbert space $H^1_{ex}=H^1_{ex}({\bf R}) 
$ with the norm} 
\[
\|w \|_{H^1_{ex}}=\sqrt{  \int_{\bf R} e^{x} w_x^2 \, dx+  \int_{\bf R} e^{x} w^2 \, dx} \;.
\]
A traveling wave solution 
 is a critical point of $J_c$, provided that $c$ is {the correct} wave speed. There is a way to determine $c$ and by the standard regularity theory, 
 {$u,v$} are $C^{\infty}$ functions on ${\bf R}$. \\ 
 
We first investigate the front propagation between the steady states $(\mu_3,\mu_3/\gamma)$ and $(0,0)$. 
{A} number of existence results for the traveling waves of (\ref{FN}) will be established when  
 at least one of the following hypotheses is satisfied.  \\ 
  
  \noindent
 ($H1$) $\quad$ $(-f'(\mu_3)-d\gamma)^2-4d>0$ \mbox{and}   $-f'(\mu_3)> d \gamma$. \\
 
   \noindent
 ($H2$) $\quad$ $(\beta-d\gamma)^2-4d>0$ \mbox{and}   $\beta> d \gamma$. \\


\noindent
Note that $f'(\mu_3^*)=f'(0)=-\beta$. {As a simple check, if $\gamma \in (\tilde{\gamma_1}, \gamma_*]$ then 
 ($H1$) implies ($H2$), while if $\gamma \geq \gamma_*$ then ($H2$) implies ($H1$).} \\ 

If $\gamma \in (\tilde{\gamma_1}, \gamma_*)$, direct calculation shows that $L_{\gamma}(\mu_2,0) > L_{\gamma}(\mu_3,0) > L_{\gamma}(0,0)$. As to 
 exhibit front propagation in both directions between 
 two distinct equilibrium states, a more difficult part is to find a traveling front 
  propagating in 
 the 
 direction that the invader $(\mu_3,\mu_3/\gamma)$
can indeed have a higher energy level than the deposed $(0,0)$. 
In the proofs, we need appropriate a priori bounds for the solutions, which can be achieved by using a truncation argument. This argument works only for 
$\beta \in(\beta_0,1/2)$ and $\gamma \in (\tilde{\gamma_2}, \gamma_*)$
{for some $\beta_0>0$ and $\tilde{\gamma_2}>\tilde{\gamma_1}$.} 
It will be shown in the Appendix that the values of $\beta_0$ and $\tilde{\gamma_2}$ do not cause severe restriction. 

 


\begin{thm} \label{mainThm}
Assume that $\beta \in(\beta_0,1/2)$, $\gamma \in (\tilde{\gamma_2}, \gamma_*)$ and ($H1$). 
There exists $d_f  > 0$ 
such that if $d 
\in (0, d_f]$ then 
(\ref{FN}) has a 
 traveling front solution 
$(c_f,u_f,v_f)$. Moreover with 
$c_f>0$ 
 the wave 
satisfies 
 \\ 
 \[
 \lim_{x \to \infty}(u_f,v_f) = (0,0)
  ~~ and ~ \lim_{x \to -\infty}(u_f,v_f) = (\mu_3,\mu_3/\gamma). 
  \] 
 \end{thm}
 
 {For the proof of Theorem~\ref{mainThm}, we seek a minimizer of a constrained variational problem. Since the argument is quite technical involved, we divide it into several steps to illustrate the scheme for showing that this minimizer is an interior critical point; that is, the constraints imposed by the admissible set are not actively engaged. To avoid conceptual 
 confusion, 
 we treat the case $\gamma \geq \gamma_*$ after the proof of Theorem~\ref{mainThm} 
 even though the same argument applies with only slight modification. 
Then an insightful observation demonstrates that there exists 
  a traveling front which propagats in the opposite direction, {stated as} follows:} 

\begin{thm} \label{opp}
For the same $\beta$ and $\gamma$ 
 as in Theorem~\ref{mainThm}, there exists $\hat d_f > 0$ 
such that if $d 
\in (0, \hat d_f]$ then 
(\ref{FN}) has a 
 traveling front solution 
$(\hat c_f,\hat u_f,\hat v_f)$ {with $\hat c_f>0$} and 
satisfying 
 \\  
 \[
 \lim_{x \to -\infty}(\hat u_f,\hat v_f) = (0,0)
  ~~ and ~ \lim_{x \to \infty}(\hat u_f,\hat v_f) = (\mu_3,\mu_3/\gamma). 
  \]
  \end{thm}

{Both} front and pulse are localized waves, the latter 
 is manifest {as} a small spot. 
Particle-like {pulses} are commonly observed in studying 
 dissipative solitons \cite{HH,L,NAY}; 
  for instance
the nerve pulses in biological systems, concentration drops in chemical systems
and {filament current} in physical systems. 
In the past considerable 
  {efforts have} been devoted to the mathematical 
    analysis for the existence of traveling waves in reaction-diffusion systems} \cite{AW,BLL,BN,CCH,CC1,CJ1,CJ2,DHK,F2,He,J,L,LMN,
LMN1,Mu,R,RTV,T1,V,VVV,Y3};
however working out the coexistence of traveling front and traveling pulse in a system with all physical parameters remain the same seems to be a quite 
 challenging 
  task. 
 Utilizing different 
 constraints in variational arguments together with 
 wave speed 
  estimates, 
 we 
  establish 
  such a 
 coexistence result: 

\begin{thm} \label{Thm2}
For the same $\beta$ and $\gamma$ 
 as 
 in Theorem~\ref{mainThm}, 
  there exists 
$d_p 
 \in (0,d_f)$ 
such that if $d 
\in (0, d_p]$ then 
(\ref{FN}) has a traveling pulse solution 
$(c_p,u_p,v_p)$ with 
{$c_p > c_f$} and 
$\lim_{|x| \to \infty}(u_p,v_p) = (0,0)$. 
 \end{thm}
 
 An important issue in 
dealing with wave propagation is the stability question. 
By virtue of existence proofs, 
 all three traveling waves 
   of (\ref{FN}) 
    are local minimizers of \eqref{def_J}. 
  An index method 
    \cite{CCHu} shows 
 that non-degenerate minimizers of \eqref{def_J} are stable traveling waves.
 {Additional numerical evidence backs up such a claim \cite{D}.}
 {Hence there is indeed a way to manipulate 
 propagation directions of a front}, which should provide an important new idea in studying dynamics of pattern formation. \\ 

Let $\delta_0=\delta_0(\beta) \equiv \frac{(1-2\beta)^2}{2}$ and 
 $\beta_1 \in (\beta,1)$ be the unique point which satisfies $\int_0^{\beta_1}f(\eta) \,d \eta=0$.
The following theorem gives a detailed description for the profile of $(u_f,v_f)$, which provides useful information to establish Theorem \ref{Thm2}. 

\begin{thm} \label{Thm3}
Let $(c_f,u_f,v_f)$ be a traveling front solution
 obtained by Theorem~\ref{mainThm}.
\begin{enumerate} 
\item[(i)] 
There exist $\zeta_M<\zeta_0<\zeta_m$ such that

(a) $u_f(\zeta_0)=0$, $u_f > 0$ on $(-\infty,\zeta_0)$,  and $u_f <0$ on $(\zeta_0,\infty)$;

(b) $u_f$ is increasing on  $(-\infty,\zeta_M)$ and it attains a unique maximum at $\zeta_M$ with $
u_f(\zeta_M)<1$;  

(c) $u_f$ attains a unique minimum at $\zeta_m$; {$u_f$ is decreasing on $(\zeta_M,\zeta_m)$ and increasing on $(\zeta_m,\infty)$;} 

(d) $v_f$ is positive and decreasing on $(-\infty,\infty)$. 


\item[(ii)] $c_f \leq \sqrt{\frac{\delta_0}{d}}$. 
\item[(iii)] If $d \to 0^+$, then $c_f \to \sqrt{\frac{\delta_0}{d}}$, $u_f(\zeta_M) \to 1$, $u_f(\zeta_m) \to 0$, {$v_f(\zeta_M) \to 0$.}
\end{enumerate}
\end{thm}

The remainder of this paper is organized as follows. Section \ref{sec_variation} begins with a variational formulation for studying the 
 traveling wave solutions of (\ref{FN}). We impose different constraints in {their respective} admissible sets {when looking} for traveling {fronts or pulses}.
 {In Section \ref{sec_min}, 
 a criteria is designed to select
out a {correct} wave speed $c_0$ 
together with a global minimizer $u_0$ extracted from {the} admissible set. 
Taking $u_0$ as a candidate of traveling wave,}  
  we need to tackle the 
task of showing the constraints imposed by the admissible set 
are not actively engaged. To reach this goal, we start with applying the corner lemma given in Section \ref{sec_corner} to conclude that $u_0$ has 
no corner point; 
that is, $u_0 \in C^1({\bf R})$. As a {principal} guideline in the proof, if 
$u_0$ would 
{touch the boundary of the admissible set,}
 we argue indirectly to rule out such possibilities by performing surgeries on $u_0$ to generate a new function $u_{new}$ within the interior of the admissible set and showing $J_{c_0}(u_{new})< J_{c_0}(u_0)$. In particular, some of such arguments rely on 
 the positivity of ${\cal L}_{c_0} u_0$ proved in Section \ref{sec_positiveV}, 
 {which, in turn, depends on the 
 behavior of this solution at 
  $+\infty$ and $-\infty$; thus 
the linearization of steady states needs to be investigated}
 in Section~\ref{sec_linearization}. With the aid of positivity of ${\cal L}_{c_0} u_0$ together with further 
 estimates, we show that $(u_0,v_0)$ satisfies (\ref{mainc}) and designate this traveling front solution by
 $(u_f,v_f)$ with speed $c_f$. 
 Then the proof of Theorem~\ref{mainThm} is complete through the analysis in Sections \ref{sec_behavior} and 
  \ref{sec_complete} with two different cases being treated separately. In addition the profile of $(u_f,v_f)$, as described in Theorem~\ref{Thm3}, will be demonstrated. 
Section \ref{sec_pulse} establishes the existence of a travel pulse solution $(u_p,v_p)$ with speed $c_p > c_f$. {Section \ref{oppf} is devoted to showing 
the second traveling front solution with propagation being opposite to the direction of $(u_f,v_f)$.} \\

To indicate which equation of a system, we use, for instance, (\ref{mainc}a) to refer the first equation of (\ref{mainc}) and (\ref{mainc}b) the second.

\section{Constrained variational problems 
} \label{sec_variation}
\setcounter{equation}{0}

\noindent
{We start with the variational framework to be used 
 for showing the existence of traveling front and traveling pulse solutions. 
 The three terms on the right hand side of (\ref{def_J}) will be referred to as the gradient term, the nonlocal term and the F-integral,
respectively to facilitate future discussion.} 
Recall from \cite{LMN} that 
\begin{equation} \label{poincare}
\mbox{if} \; w \in H^1_{ex} \;\mbox{then} \quad
\left\{
\begin{array}{rl}
\frac{1}{4} \int_{\bf R} e^{x} w^2 \, dx & \leq  \int_{\bf R} e^{x} w_x^2 \, dx  \;, \\ \\
e^{x} w^2(x) & \leq  \int_x^{\infty} e^{y} w_y^2\, dy \;.
\end{array} \right.
\end{equation}
Since
\begin{equation} \label{equivalent}
\int_{\bf R} e^{x} w_x^2 \, dx \leq  \| w \|_{H^1_{ex}}^2 \leq 5
 \int_{\bf R} e^{x} w_x^2 \, dx,
\end{equation}
{$\| w_x \|_{L^2_{ex}}$ will be taken
 as an equivalent
  norm of $H^1_{ex}$.} \\ 

{We now examine the nonlocal term.} For $c>0$, 
the solutions of the characteristic 
equation $c^2 r^2+c^2 r-\gamma=0$ are
\begin{equation} \label{r21}
r = \frac{1}{2c}(-c \pm \sqrt{c^2+4 \gamma} ) \;,
\end{equation}
which {are} 
 denoted by $r_1$, $r_2$ with $r_1<-1<0<r_2$. Then 
  (\ref{v_Nsoln}) is rewritten as
\begin{equation} \label{Lu}
{\cal L}_cu \;(x)= \frac{e^{r_1 x}}{c \sqrt{c^2+4\gamma}}\int_{-\infty}^x e^{-r_1 s}\; u(s) \, ds+
\frac{e^{r_2 x}}{c \sqrt{c^2+4\gamma}}\int_{x}^{\infty} e^{-r_2 s}\; u(s) \, ds \;.
\end{equation}
Direct calculation yields 
 \begin{equation} \label{vuineq}
  \int_{\bf R} e^{x} (c^2 v'^2+ \gamma v^2)\, dx = \int_{\bf R} e^{x} uv \, dx \leq \| u\|_{L^2_{ex}} \, \| v\|_{L^2_{ex}}.
 \end{equation} 
 It is known 
 \cite{CC1} that
 \begin{equation} \label{v_bound}
 \left\{ \begin{array}{l}
 \|v\|_{L^2_{ex}} \leq \frac{4}{ c^2} \|u\|_{L^2_{ex}} \; \\ \\
 \| v'\|_{L^2_{ex}} \leq  \frac{2}{c^2} \| u\|_{L^2_{ex}} \; \\ \\
 0 \leq \int_{\bf R} e^{x}  \{c^2( {\cal L}_c u)'^2+\gamma ({\cal L}_cu)^2\}\, dx =
 \int_{\bf R} e^{x} u \, {\cal L}_c u \, dx \;
 \end{array} \right.
 \end{equation}
and 
 \begin{equation} \label{v_H1}
\| {\cal L}_c u \|_{H^1_{ex}} \leq \frac{2\sqrt{5}}{c^2} \| u \|_{L^2_{ex}} \;.
\end{equation} 
We state a list of 
 properties as simple consequences from direct calaulation.
 \begin{lemma} \label{lem_positive}
If $w \in L^2_{ex}$ then $\int_{\bf R}e^{x}  w \, {\cal L}_c w \, dx \geq 0$.
\end{lemma}
\begin{lemma} \label{lem_translateNonlocal}
For $a \in {\bf R}$, 
 ${\cal L}_c (w(\cdot-a))
=({\cal L}_c w)(\cdot-a)$.
\end{lemma}

\begin{lemma} \label{lem_energyTranslate}
If $w \in H^1_{ex}$ then 
$J_c(w(\cdot-a))=e^a J_c(w)$.
\end{lemma}





{To establish multiple traveling wave solutions for the same parameters, we need a mean to differentiate one from the others.}
The following definition describes certain 
 {oscillation} constraints to be added into a %
  class of admissible functions for $J_c$. \\  


\noindent
{\em Definition:}
A continuous function $w$ is in the class $+/-$, if there exists $-\infty \leq x_1  \leq \infty$ such that $w \geq 0$ on $(-\infty,x_1]$,
and $w \leq 0$  on $[x_1,\infty)$.
\noindent
\begin{remark} \label{remark1}
(a) In the above definition, the choice of $x_1$ is not necessarily unique.
In case $x_1=-\infty$, then $w \leq 0$ on the real line; similarly $x_1=\infty$ indicates that $w \geq 0$ on the real line. Both examples 
are included in the class $+/-$. \\
(b) A function $w$ is said to change sign once, if 
$w \geq 0$ on $(-\infty,x_1]$, $w \leq 0$ on $[x_1,\infty)$ and $w\not \equiv 0$ in each interval. In this case, for convenience in notation, we stipulate $x_1=\sup\{x:  w(x)>0\}$, and it will be referred to as a crossing point.
\\  
(c) A typical function in the class $+/-$ is illustrated in Figure~\ref{fig1}. 
\end{remark}

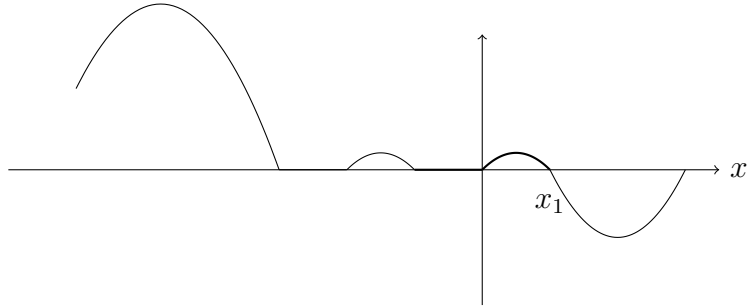
\begin{figure}[ht]
\centering
$$
\begin{tikzpicture}[domain=-7:4,smooth,color=black,scale=.9]

\draw[->] (-7,0) -- (3.5,0) node[right] {$x$};
\draw[->] (0,-2) -- (0,2); 

\clip (-7,-3) rectangle (3,3);

\draw[domain=-6:-3,smooth,variable=\t,color=black]
plot ({\t},{-.8*(\t+6.5)*(\t+3)});

\draw[domain=-3:-2,smooth,variable=\t,color=black]
plot ({\t},{0});

\draw[domain=-2:-1,smooth,variable=\t,color=black]
plot ({\t},{-1*(\t+2)*(\t+1)});

\draw[domain=-1:0,thick,smooth,variable=\t,color=black]
plot ({\t},{0});

\draw[domain=0:1,thick,smooth,variable=\t,color=black]
plot ({\t},{-\t*(\t-1)});

\draw[domain=1:3,smooth,variable=\t,color=black]
plot ({\t},{(\t-1)*(\t-3)});

\node at (1,-.5) {$x_1$}{};


\end{tikzpicture}
$$
\caption{A function in the class $+/-$.}
\label{fig1}
\end{figure}

The potential $F$ has exactly two local minima, namely 
$F(0)=0$ and $F(1)=-\frac{1-2\beta}{12}$. 
There exist $\beta_1 \in (\beta,1)$ and $\tilde{\beta_2} > 1$ such that 
$F({\beta_1})=F(\tilde{\beta_2})=0$. 
{Choose a small positive number $\theta_1$  
such that  
\begin{equation} \label{theta1}
\beta<\beta_1<1< 1+\theta_1 = \beta_2 < \min \{\tilde{\beta_2}, 1+\beta/2 \}.
\end{equation}
To get 
 a priori bounds for the solutions, we use a truncation argument. Observe that 
$v=f(\mu_3)+f'(\mu_3) (u-\mu_3)$ is the line tangent to the graph $v=f(u)$ at the point $(\mu_3,\mu_3/\gamma)$; in fact 
 \begin{equation} \label{t2a}
f(\mu_3)+f'(\mu_3) (\xi-\mu_3) \geq f(\xi) \quad \mbox{for} \; \xi \in [-M_1,\beta_2]
\end{equation}
is a sufficient condition to work out the required truncation argument. We remark that $M_1$ is a positive number. The detailed derivation of \eqref{t2a} is given in the Appendix, involving the choice of $\beta_0$ as well as $\tilde{\gamma_2}$. } \\


{In Theorem \ref{mainThm} we study the existence of 
  traveling front  solutions; the class of admissible functions to be 
employed 
  is} 
 \begin{eqnarray} \label{def_A}
{{\cal A}_f} &\equiv &  \{ w \in H^1_{ex}({\bf R}):  \int_{\bf R} e^{x} w_x^2\, dx =2, \;
-M_1 \leq w \leq \beta_2, \;
 w \; \mbox{is in the class} \; +/-, \nonumber \\
&& \qquad \qquad \qquad {w-\mu_3 \; \mbox{is in the class} \; +/-}\;
 \} ,
\end{eqnarray}
where $\theta_3$ is a small positive number. The imposed condition 
\[ 
w \; \mbox{is in the class} \; +/- \nonumber \\
\]
is referred to as an {oscillation} constraint, so is the one imposed on $w-\mu_3$. 
With (\ref{mainc}) being an autonomous system, the constraint $\int_{\bf R} e^x u_x^2\, dx=2$ is imposed 
 to eliminate a continuum of solutions due to translations invariance. 
Once  the value of $c$ 
 is determined, we seek  
  a minimizer of ${J}_c$ over ${\cal A}_f$ as a traveling front solution of (\ref{FN}). 
\\


For the traveling pulse solution, we impose different {oscillation} constraints 
 on the admissible set.  \\

\noindent
{\em Definition:}
A continuous function $w$ is in the class $-/+/-$, if there exist $-\infty \leq x_1 \leq
x_2 \leq \infty$ such that $w \leq 0$ on $(-\infty,x_1] \cup [x_2,\infty)$,
and $w \geq 0$  on $[x_1,x_2]$.\\

\noindent
The admissible set for traveling pulse solution is
\begin{eqnarray*} 
{\cal A}_p =\{ w \in H^1_{ex}({\bf R}): &&  \int_{\bf R}e^x w_x^2 \, dx=2, -M_1 \leq w \leq \beta_2,
 \; w \; \mbox{is in the class} \; -/+/-\}. 
\end{eqnarray*}
{It is clear that ${\cal A}_f \subset {\cal A}_p$.}\\



As 
 in \cite{CC1}, the following lemma indicates  that  $J_c$ is well defined. 

\noindent
\begin{lemma} \label{lem_Jbound}
For every $c>0$, 
there exist 
$m_1=m_1(\beta)>0$ and  $m_2=m_2(\beta,\gamma)>0$
such that 
\begin{equation} \label{J_bound}
-m_1 \leq {J}_c(w) \leq dc^2 +  m_2 (1+\frac{1}{c^2})
\end{equation}
if $w \in {\cal A}_f$ or ${\cal A}_p$.
\end{lemma}

 \vspace{.2in}
\noindent


\section{Existence of a minimizer} \label{sec_min}
\setcounter{equation}{0}

In recent years remarkable advances have been made on the 
 variational methods to study the traveling wave solutions. 
 An important step working towards the existence 
  is how to determine the wave speed. 
 In $\cite{LMN}$ Lucia, Muratov and Novaga worked on a 
  constrained variational problem and the wave speed was 
 obtained by a rescaling argument. 
 Heinze $\cite{He}$ 
  imposed an equivalent 
  ansatz 
  for the change of variables in working with a constrained variational problem, 
  in this approach the wave speed is given by 
 a Lagrange multiplier. 

In \cite{CC1} we 
employed a slightly different process 
  to determine the speed of traveling wave solution. 
Set ${\cal J}_f(c) \equiv \inf_{w \in {\cal A}_f} J_c(w)$. {It will be seen that the speed 
 of a traveling front is a solution 
of ${\cal J}_f(c)=0$.
  Similarly we work with
${\cal J}_p(c) \equiv \inf_{w \in {\cal A}_p} J_c(w)$ for traveling pulse solutions. 
All the lemmas in this and several 
 later sections are valid for both traveling front and pulse solutions; 
to avoid unnecessary duplication, we first 
state the results 
 for the fronts and will point out the discrepancy 
in section~\ref{sec_pulse} as to 
 study the traveling pulse solutions.} 
The following 
 estimates will be used 
  get 
 a solution 
of ${\cal J}_f(c)=0$. The proofs of Lemma~\ref{lem1} and Lemma~\ref{lem_largec} can be found in \cite{CC1}.

\begin{lemma} \label{lem1} 
If
$C_1 \leq w \leq C_2$ then
$\frac{C_1}{\gamma} \leq {\cal L}_{c_0} w \leq \frac{C_2}{\gamma}$. 
\end{lemma}

\begin{lemma} \label{lem_largec}
There exists a $\bar{c}=\bar{c}(d,\beta)>0$ such that if $c \geq \bar{c}$ then
${\cal J}_f(c)>0$. 
\end{lemma}

\vspace{.1in}
Next we show that ${\cal J}_f(c)<0$ if $c=\sqrt{\frac{24}{1-2\beta}}$
and $d$ is sufficiently small. 

\begin{lemma} \label{lem_test_fun}
{Let 
{$\gamma>\tilde{\gamma_1}$} and $c_* \geq 0$ be given. If $\tilde{c} 
=\sqrt{\frac{24}{1-2\beta}}+ c_*$ there exists  $\tilde{d_0}=\tilde{d_0}(c_*)$} 
such that for every $d \in(0, \tilde{d_0}]$, a function $w_d$ 
can be selected out from ${\cal A}_f$ {to satisfy $J_{\tilde{c}}(w_d)<0$.}
\end{lemma}

\begin{proof} Set $a=\sqrt{d}$, 
which is a small number.
Define {a continuous function}
\begin{equation}
w(x) = \left\{ \begin{array}{ll} 
    1, & \mbox{if} \; x \leq  0, \\
    - \frac{e^{-a}}{1-e^{-a}}+\frac{1}{1-e^{-a}} e^{-x}, & \mbox{if} \; 0 <x <a, \\
    0, & \mbox{if} \; x>a. 
    \end{array} \right.
\end{equation}
If $c= \tilde{c}$, 
 straightforward calculation yields 
 \begin{eqnarray} 
J_{\tilde{c}}(w)
&\leq& - \frac{(1-2\beta)}{24} +{(1+\frac{\tilde{c}^2}{2}) O(\sqrt{d})} \label{negativeJc} \;.
\end{eqnarray}
Hence there is a $\tilde{d_0}=\tilde{d_0}(c_*)$ such that $J_{\tilde{c}}(w)<0$ for any 
$d \leq \tilde{d_0}$.
Since $\int_{\bf R}e^x w'^2 \,dx=\frac{1}{1-e^{-a}}=\frac{1}{\sqrt{d}} (1+O(\sqrt{d}))$, there is a $b\in{\bf R}$ such that 
$w_{d}=w(\cdot+b) \in {\cal A}_f$ and $J_{\tilde{c}}(w_{d})=e^{-b} J_{\tilde{c}}(w)<0$.
We refer to \cite{CC1} for the detailed calculation. 
\end{proof}

Taking $d_0 \equiv  \tilde{d_0}(0)$ and $c_*=0$ in Lemma \ref{lem_test_fun} gives 
{\begin{cor} \label{cor_test_fun} 
Let {$\gamma>\tilde{\gamma_1}$ and} $d \leq d_0$. 
If $\underline{c}= \sqrt{\frac{24}{1-2\beta}}$ 
then $\inf_{{\cal A}_f}J_{\underline{c}} <0$.
\end{cor}
}

\begin{lemma} \label{lem_min_cont} 
  ${\cal J}_f:  [\underline{c},\bar{c}] \to {\bf R}$  is a continuous function of $c$.
\end{lemma}

The 
 continuity of ${\cal J}_f$ as well as 
  the next lemma has been proved 
   in \cite{CC1}.

\begin{lemma} \label{lem_L_cont1} 
Let ${c}>0$ and 
$\{c_n\}_{n=1}^{\infty} \subset (c/\sqrt2,c\sqrt3 /\sqrt2)$ such that $c_n \to c$ as $n\to \infty$.
Then
there exists a positive constant ${M_5=M_5(c)}=20/c^4$ such that
\begin{equation} \label{v_diff_c1}
\| {\cal L}_{c_n}w - {\cal L}_{c}w \|_{H^1_{ex}} \leq  M_5\,  |c_n^2 -c^2| \; \|w\|_{L^2_{ex}} \;
\end{equation}
for all $w \in H^1_{ex}$. 

\end{lemma}

Since ${\cal J}_f(\underline{c})<0$ and ${\cal J}_f(\bar{c})>0$, 
  a solution 
of ${\cal J}_f(c)=0$ immediately follows from Lemma~\ref{lem_min_cont}.

\begin{cor} \label{cor_zeroJ} 
There exists $c_0 \in [\underline{c},\bar{c}]$ such that 
${\cal J}_f(c_0)=0$.
\end{cor}





 To show the existence of a traveling front solution, 
  we seek
  a minimizer of $J_{c_0}$ in ${{\cal A}_f}$. This requires certain 
   estimates which can be obtained by those arguments used in \cite{CC1}. 

\begin{lemma} \label{lem_b} 
Let  $b_{0}=-2 \log (\beta/\sqrt{2}) >0$.   
If $w\in {{\cal A}_f}$ then $|w(x)| \leq \beta$ for all $x \in [b_{0},\infty)$. 
\end{lemma}


\begin{lemma} \label{lem3} If {$\gamma>\tilde{\gamma_1}$ and} $d \leq d_0$, 
there is a minimizer $u_0 \in {\cal A}_f$ such that $J_{c_0}(u_0)=0$. 
\end{lemma}

\noindent
\begin{proof}
Let $\{ w^{(n)} \}_{n=1}^{\infty} \subset {\cal A}_f$ be a minimizing sequence of
$J_{c_0}$.  
From $\int_{\bf R} e^{x} (w_x^{(n)})^{\,2}\, dx =2$
and (\ref{equivalent}), 
 a uniform bound on $\| w^{(n)}\|_{H^1_{ex}}$ gives 
 a $W \in H^1_{ex}$ such that along 
a subsequence 
 $w^{(n)} \rightharpoonup W$ weakly in $H^1_{ex}$ and strongly in {$L^\infty_{loc}(\bf R)$}. 
 Consequently $-(M_1{+\theta_3}) \leq W \leq \beta_2$.
With slight modification, the rest of the proof is similar to that of 
 Lemma 4.2 of \cite{CC1}. 
\end{proof}

 \noindent
\begin{remark} \label{remark3}
(a) At this moment {whether the minimizer $u_0$ is 
a traveling wave solution of (\ref{FN}) is not yet known,}
due to the 
constraints imposed on the admissible set ${\cal A}_f$. \\
(b) {If 
$w \in H^1_{ex}\setminus \{0\}$ {and satisfies all the constraints in the admissible set ${\cal A}_f$
except the condition $\int_{\bf R} e^x w_x^2 \, dx=2$}, then
 $J_{c_0}(w) \geq 0$; indeed, 
 by taking a suitable $a \in {\bf R}$,
the function $W(x) \equiv w(x-a)$ satisfies $\int_{\bf R}e^x W_x^2\, dx=2$ 
and $J_{c_0}(W)=e^a J_{c_0}(w)$. 
Since $W \in {\cal A}_f$, it follows that $J_{c_0}(W) \geq 0$.} \\  
(c) From (b), it is clear that $u_0$ also minimizes ${J}_{c_0}$ on the set 
\begin{eqnarray*}
\hat {\cal A}_f &\equiv &
\{ w \in H^1_{ex}({\bf R}):  1 \leq \int_{\bf R} e^x w_x^2 \, dx \leq 3, \; -M_1 \leq w \leq \beta_2, \; {w \; \mbox{is in the class} \; {+/-}}, \\
&& \quad \quad \quad {w-\mu_3 \; \mbox{is in the class} \; +/-}\} \;.
\end{eqnarray*}
Hence in the derivation of 
the Euler-Lagrange equation associated with $J_{c_0}$, the integral constraint imposed on ${\cal A}_f$ does not result in a 
Lagrange multiplier. Furthermore, $u_0$ satisfies (\ref{mainc}a) at those points where
$u_0 \ne -M_1,0,\mu_3$ 
or $\beta_2$, while $v_0 \equiv {\cal L}_{c_0}u_0$
satisfies (\ref{mainc}b) on $\bf R$. Since $u_0$ is uniformly bounded on $\bf R$, so is $v_0$ {by Lemma~\ref{lem1}.} \\ 
\end{remark}


Next we need to tackle the 
task of showing the constraints imposed by 
${\cal A}_f$ are not actively engaged. 
This will be accomplished by using indirect arguments; that is, 
we draw a contradition by performing surgeries on $u_0$ to generate a new function $u_{new}$ within ${\cal A}_f$ such that 
  $J_{c_0}(u_{new})< J_{c_0}(u_0)$ holds. 
 The following lemma 
 is a 
 useful formula in calculating the change of 
 the nonlocal term. 

\begin{lemma} \label{lemDiff}
 Suppose that $u_{new} \in H^1_{ex}$, a function obtained by making changes on $u_0$. 
 Then the change in 
 the nonlocal term is
 \begin{equation} \label{changeNonlocal}
\frac{1}{2} \int_{\bf R} e^{x} (u_{new} \, {\cal L}_{c_0} u_{new} -u_0 \, {\cal L}_{c_0} u_0 ) \, dx=
\frac{1}{2} \int_{\bf R} e^{x} (u_{new}-u_0) \, {\cal L}_{c_0} (u_{new} +u_0)\, dx\;.
\end{equation}
\end{lemma}

\noindent
\begin{proof} It is a direct consequence of the fact that ${\cal L}_{c_0}$ is self adjoint with respect to 
the 
$L^2_{ex}$ inner product.
\end{proof}

 \noindent
\begin{remark} \label{remark4}
 If the support of $u_{new}-u_0 \,$ lies inside a finite interval $[a,b]$, then
 the change in the nonlocal 
 term is actually 
  calculated within the same interval $[a,b]$.
 This turns out to be especially fruitful to deal with the case that 
 $u_{new}-u_0$ is small, as ${\cal L}_{c_0} u_{new}$ {will be close} to ${\cal L}_{c_0} u_0$ on $[a,b]$.
\end{remark} 
Several estimates stated in the next lemma have been 
 established in \cite{CC1}; it shows $u_0> -(M_1{+\theta_3})$ which releases one of the constraints imposed on ${\cal A}_f$. 
 {For this proof, it needs only the convexity of $F(\xi)$; i.e., $f'(\xi)<0$, for $\xi \leq 0$.}

\begin{lemma} \label{lemCutoff1}  
 If $u_0$ is a minimizer of ${J}_{c_0}$ then 

\noindent
(a) $\inf u_0 \geq -M_1$,

\noindent
(b) $
\sup u_0> \beta_1$ and $|\{ x: u_0(x) > \beta_1\}|
\geq {6 d c_0^2 \beta^2/(1-2\beta)}$, 

\noindent
(c) 
$d c_0^2 \leq(1-2\beta)/6\beta^2$. 
\end{lemma}


Away from the subset where $u_0$ equals $0$, {$\mu_3$} or $\beta_2$, there is a room for this minimizer 
 to be
perturbed by $C^{\infty}_0$ functions
so that the perturbed function still lies
inside $\hat {\cal A}_f
$. In view of Remark~\ref{remark3}(c), the next lemma follows from 
 the derivation of 
the Euler-Lagrange equation associated with $J_{c_0}$ and 
 standard regularity bootstrap argument. 


\begin{lemma} \label{lem_Bmin}
Let 
 $u_0$ be a minimizer of ${J}_{c_0}$ and $v_0 = {\cal L}_{c_0}u_0$. 
 Then 
$u_0$ satisfies (\ref{mainc}a) at those points where
$u_0 \ne 0$, {$u_0 \ne \mu_3$} and $u_0 \ne \beta_2$, while $v_0 
 \in C^2({\bf R})$
and it satisfies (\ref{mainc}b) everywhere.
\end{lemma}

{As a remark,
most of lemmas established in this section can be used to 
study the traveling pulse solutions 
 as well.}





\section{
{Continuity of $u'_0$}} \label{sec_corner}
\setcounter{equation}{0}

{As in Lemma~\ref{lem_Bmin}, 
 $u_0$ always stands for a minimizer of $J_{c_0}$ in ${\cal A}_f$ 
and $v_0={\cal L}_{c_0}u_0$. We keep this notation without further comment.} 
Due to the constraints imposed in ${\cal A}_f$,
there is a possibility that $u_0 \equiv \beta_2$, {$u_0 \equiv \mu_3$} or $u_0 \equiv 0$ 
on some 
 subintervals of $(-\infty, \infty)$; 
  {we thus cannot} conclude that $u_0$ 
satisfies (\ref{mainc}a) on such intervals including their boundary points. This scenario {has to} be eliminated in order to {show that
$u_0$ represents a traveling wave solution.}


Our investigation starts with the derivatives of 
$u_0$.  

\begin{lemma} \label{lem_c1}
Suppose that $u_0(x_0)=\beta_2$ and, for some $\ell>0$,
 $\mu_3 < u_0(x) < \beta_2$ for $x \in [x_0-\ell, x_0)$. Then both $\lim_{x \to x_0^-} u_0'(x)$ and
$\lim_{x \to x_0^-} u_0''(x)$ exist. In addition, $u_0$ is a $C^{\infty}$-function
on $[x_0-\ell,x_0]$ and it satisfies (\ref{mainc}a) on
$[x_0-\ell,x_0]$. 
The case of $(x_0,x_0+\ell]$ is parallel. 
Also, 
dealing with $u_0(x_0) = 0$ or $u_0(x_0) = \mu_3$ is analogue.. 
\end{lemma}

The next lemma, which will be referred to as the corner lemma, enables us to eliminate the
possibility that a sharp corner appears on the graph of $u_0$. {Note that it
does not require  
$u_0$ satisfy (\ref{mainc}a) on either $[x_0-\ell,x_0]$ or $[x_0,x_0+\ell]$.} 

\begin{lemma} \label{lem_corner1}
{Suppose} $u_0(x_0) = \beta_2$ and $u_0 \in C^1[x_0-\ell,x_0]
\cap C^1[x_0,x_0+\ell]$ {for some $\ell>0$}, then $\lim_{x \to x_0^-}u_0'(x) = \lim_{x_0
\to x_0^+} u_0'(x)$. 
The same assertion holds in case $u_0(x_0)=0$ {or $u_0(x_0)=\mu_3$}.
\end{lemma}
\noindent
We refer to \cite{CC1} for the proofs of Lemma \ref{lem_c1} and Lemma \ref{lem_corner1}. \\

{Let $\nu \in \{0,\mu_3,\beta_2\}$ and suppose that $u_0(x_0)=\nu$. Due to the constraints 
 imposed on ${\cal A}_f$, $u_0$ may not satisfy (\ref{mainc}a) at $x_0$, since 
on an interval containing $x_0$ as an interior point, $u_0$ cannot always stay 
  inside of ${\cal A}_f$ even when 
  an arbitrary 
  small perturbation is taken within this 
 interval. } 
  In such a situation, $x_0$ is referred to as a constrained 
  point against perturbation or simply called a constrained 
   point. The same can happen to a minimizer in ${\cal A}_p$, as to be treated for studying traveling pulse solution. \\

{We now investigate a constrained 
 point $x_0$ of $u_0$, if it exists. 
  Let us consider only the case that $u_0(x_0)=\nu$ and $u_0 \in C^{1}[x_0-\ell,x_0]$ for some $\ell>0$, since the argument is not different in dealing with 
   $u_0 \in C^1[x_0,x_0+\ell]$.} 
By taking $\ell$ smaller if necessary, eventually we encounter three types of 
 behavior 
  on the left side 
of $x_0$:

\noindent
(P1a) $u_0<\nu$ on $[x_0-\ell, x_0)$;

\noindent
{(P1b) $u_0>\nu$ on $[x_0-\ell,x_0)$; (
this 
does not happen if $\nu=\beta_2$)}

\noindent
(P2) {$u_0 = \nu$} on $[x_0-\ell,x_0]$; 

\noindent
(P3) There exist $a_1<b_1 < a_2<b_2< a_3<b_3 \dots$, two sequences of points in 
 $[x_0-\ell,x_0]$,  
such that 
\[ \left\{ \begin{array}{l}
u_0 \quad \mbox{satisfies (\ref{mainc}a) on each interval} \; (a_i,b_i), \; i=1,2, \dots \;, \\
u_0=\nu \quad \mbox{on}\; [x_0-\ell,x_0] \setminus \cup_{n=1}^{\infty} (a_n,b_n) \;,
\end{array} \right.
\]
with both $a_n \to x_0^-$ and $b_n \to x_0^-$. \\

In case of {(P1a), (P1b)} or (P2), 
$u_0 \in C^{1}[x_0-\ell,x_0]$ follows from Lemma~\ref{lem_c1} and Lemma~\ref{lem_corner1}. 
Next we show that 
 $u_0 \in C^1({\bf R})$ in all cases; (P1) to (P3). 

\noindent
\begin{lemma} \label{lem_case3} 
If $x_0$ is a n accumulation point as stated in (P3),  
then 
{$u_0 \in C^1[x_0-\ell,x_0]$ with}
 $u_0'(x_0)=v_0'(x_0)=0$. Moreover \\ 
(a) if 
$u_0(x_0)=\beta_2$ then
 $v_0(x_0)=f(\beta_2)<0$; \\
(b) if 
$u_0(x_0)=0$ then 
 $v_0(x_0)=f(0)=0$;\\
{(c) if  
$u_0(x_0)=\mu_3$ then 
 $v_0(x_0)=f(\mu_3)>0$.}

\end{lemma}

\begin{proof} 
Without loss of generality we may assume that $x_0-a_1 \leq 1$. On the interval $[a_n,b_n]$, $c_0^2 (e^x u_0')'=- e^x (f(u_0)-v_0)$ and there is a $s_n \in (a_n,b_n)$ such that $u_0'(s_n)=0$. Since 
$\|e^x (f(u_0)-v_0)\|_{L^{\infty}(a_n,b_n)} \leq C_1 e^{x_0+1}$ for some constant $C_1$ not depending on $x_0$ or $n$, a simple integration yields
$c_0^2 e^x |u_0'(x)| \leq C_1 e^{x_0+1} |\int_{s_n}^x dt |$, which gives
$|u_0'(x)|  \leq C_1 e^{2} (b_n-a_n)/c_0^2$.
As $n \to \infty$, it follows from $|b_n-a_n| \to 0$ that $\| u_0'\|_{L^{\infty}(a_n,b_n)} \to 0$.  
Then $u_0 \in C^1[a_1,
x_0]$ if we set $u_0'(x_0^-)=0$.
Since the same argument shows $u_0 \in C^1[x_0,x_0+\delta]$, invoking Lemma~\ref{lem_corner1} yields $u_0 \in C^1[x_0-l,x_0+l]$ for some $l>0$. This completes the proof of $u_0 \in C^1({\bf R})$.

{Next consider the case that $u_0(x_0)=\beta_2$.
Since $u_0'(x_0)=0$} and $u_0 \leq \beta_2$ everywhere,
by (\ref{mainc}a)
\begin{equation} \label{zn}
f(u_0(s_n))-v_0(s_n)= -d c_0^2 u''_{0}(s_n) \leq 0 \;.
\end{equation}
 As $s_n \to x_0^-$, it follows that
\begin{equation} \label{4.6}
f(u_0(x_0))-v_0(x_0) \leq 0 \;.
\end{equation}

Clearly $u_0 \in C^2[a_n,b_n]$ by Lemma~\ref{lem_c1}. In view of  
$u_0(b_n)=\beta_2$
and $u'_0(b_n)=0$, we know $u''_0(b_n) \leq 0$ and consequently
\begin{equation} \label{bn}
f(u_0(b_n))-v_0(b_n)=-u''_0(b_n) \geq 0 \;.
\end{equation}
Passing to a limit as $n \to \infty$ gives
$f(u_0(x_0))-v_0(x_0) \geq 0$. 
This together with (\ref{4.6}) yields $f(u_0(x_0))-v_0(x_0) = 0$; in other words, $v_0(x_0)=f(\beta_2)<0$.

We claim $v_0'(x_0)=0$. First suppose that $v_0'(x_0)<0$. This together with $u_0'(x_0)=0$ gives 
$(f(u_0)-v_0)'|_{x=x_0}=-v_0'(x_0)>0$. Since $f(u_0)-v_0=0$ at $x=x_0$,
it follows that $f(u_0(x))-v_0(x)<0$ on an interval $[x_0-\delta,x_0)$ for some $\delta>0$. 
This is incompatible with (\ref{bn}). Similarly $v_0'(x_0)>0$ would  contradicts (\ref{zn}).
Now the claim is justified, so the proof of (a) is complete. \\ 

{The proof for the case $u_0(x_0)=0$ or {$u_0(x_0)=\mu_3$} is slightly different,  
since $u_0$  
can cross $0$ or {$\mu_3$} in
$(a_1,x_0)$; nevertheless due to both $u_0$ {and $u_0-\mu_3$} are in the class {$+/-$} 
, by choosing $a_1$ sufficiently close to $x_0$ then $u_0$ 
does not change sign in $[a_1,x_0]$. 
The rest of the proof is similar to that of (a). We omit the detail.}
\end{proof}

In summary, should $u_0$ have a constrained 
 point $x_0$ {of the types (P2) or (P3)}, the corner lemma together with Lemma~\ref{lem_case3} shows that 
 $u_0'(x_0)={v_0'(x_0)}=0$.


\section{
{Linearization at $(0,0)$ and $(\mu_3,\mu_3/\gamma)$}} \label{sec_linearization}
\setcounter{equation}{0}

\vspace{.1in}
 As to work towards 
  a traveling front solution, a crucial step 
 is to show the positivity of ${\cal L}_{c_0} u_0$.  
Reinforced by this fact, Lemma~\ref{lemDiff} turns out to be more helpful, since the sign of the change in (\ref{changeNonlocal}) will be easier to evaluate. To reach this goal, we extract further information for 
$(u_0,v_0)$ from the linearization of (\ref{mainc}) at $(0,0)$ as well as at $(\mu_3,\mu_3/\gamma)$. 
It is assumed that ($H1$) is satisfied. As a remark, only ($H2$) is required for carrying out this analysis about $(0,0)$. \\ 



Define $g_1(u) \equiv f(u)+\beta u  =u^2(1+\beta-u)$ and rewrite  (\ref{mainc}) as 
\begin{equation} \label{main1}
c_0^2 \vectwo{{u_0}}{{v_0}}_{xx}+c_0^2 \vectwo{u_0}{v_0}_x- A \vectwo{{u_0}}{{v_0}}= - \frac{1}{d} \vectwo{g_1(u_0)}{0},
\end{equation}
where
\[
A=\mattwo{\frac{\beta}{d}}{\frac{1}{d}}{-1}{\gamma} \;.
\]
{We begin with the calculation on the
eigenvalues, eigenvectors and left   
eigenvectors of $A$.} \\

The eigenvalues $\lambda_1$, $\lambda_2$ of $A$ satisfy
\begin{equation} \label{eigenvalue1}
d \lambda^2 - (\beta + d \gamma) \lambda + 1+\gamma \beta = 0.
\end{equation}
They are distinct {positive} numbers because 
($H2$) holds.
Since 
 \begin{equation} \label{trace}
 \lambda_1+\lambda_2= trace(A)=\gamma+\beta/d,
  \end{equation}
   it follows that 
 \begin{equation} \label{rankEvalue}
 \lambda_1<\frac{1}{2}{(\gamma+\frac{\beta}{d})}< \lambda_2
  \;.
 \end{equation} 
{Writing  
\[
0=d \lambda^2 - (\beta + d \gamma) \lambda + 1+\gamma \beta=d (\lambda- \frac{\beta+d\gamma}{2d})^2+ \frac{4d-(\beta-d \gamma)^2}{4 d}
 \]
 and invoking ($H2$), we see 
  that $d \lambda^2 - (\beta + d \gamma) \lambda + 1+\gamma \beta>0$ if $\lambda \geq \beta/d$. This enforces
 $\lambda_2< \beta/d$.}

{For each eigenvalue $\lambda_i$, denoted by
${\bf a}_i$ an associated eigenvector for $A$ and ${\bf l}_i=(1,\eta_i)$ {its left eigenvector}. 
As $\lambda_1 \ne \lambda_2$ it is known  that ${\bf l}_1 \cdot {\bf a}_2=
 {\bf l}_2 \cdot {\bf a}_1=0$, we may take ${\bf a}_1=(\eta_2,-1)^T$ and ${\bf a}_2=(\eta_1,-1)^T$. Since the first row of the matrix $A-\lambda_i I$ is
$(\beta/d-\lambda_i  \;, 1/d)$, both $\eta_1$ and $\eta_2$ must be positive. 
The second row of the matrix
$A-\lambda_i I$ is $(-1, \gamma-\lambda_i)$, which gives 
$\lambda_i > \gamma$. Thus
\begin{equation} \label{rankEigenvalue}
0<\gamma<\lambda_1<\frac{1}{2}{(\gamma+\frac{\beta}{d})}< \lambda_2< \beta/d \;.
\end{equation} }
{A direct calculation from the first row of
\begin{equation*}
(A^T-\lambda_iI)\left(\begin{array}{c} 1 \\ \eta_i \end{array}\right)
=\left(\begin{array}{c} 0 \\ 0 \end{array}\right)
\end{equation*}
gives 
\begin{equation} \label{da1}
\beta/d-\lambda_i=\eta_i.
\end{equation}
Hence
\begin{equation} \label{partd}
{\bf l}_1 \cdot {\bf a}_1= \eta_2-\eta_1 =\lambda_1-\lambda_2 <0 \; \quad
\mbox{and}  \quad {\bf l}_2 \cdot {\bf a}_2=\eta_1-\eta_2>0 \;.
\end{equation}}

%


Suppose that $(u_0,v_0)$ 
satisfies (\ref{mainc}) and $(u_0,v_0) \to (0,0)$ as $x \to \infty$. 
As $g_1(u)=O(u^2)$ for small $u$, the dominant
behavior of $(u_0,v_0)$ can be analyzed by 
linearizing (\ref{main1}) about $(u,v)=(0,0)$:
\begin{equation} \label{linearized}
c_0^2 \vectwo{\tilde{u}}{\tilde{v}}_{xx}+ c_0^2 \vectwo{\tilde{u}}{\tilde{v}}_x- A \vectwo{\tilde{u}}{\tilde{v}}={\bf 0}\quad \mbox{on} \; [a_1,\infty)\;,
\end{equation}
provided that $a_1$ is sufficiently large. Denoted by $s_2,s_3$, the roots of $c_0^2 s^2+c_0^2 s -\lambda_1=0$  and $s_1,s_4$, those of
$c_0^2 s^2+c_0^2 s -\lambda_2=0$. By plotting the curve $z=s(s+1)$ and
the horizontal lines $z=\lambda_1/c_0^2$ and $z=\lambda_2/c_0^2$, 
it is readily seen that
\begin{equation} \label{4eigen}
s_1< s_2<-1<0< s_3<s_4, 
\end{equation}
since $\lambda_2>\lambda_1>0$. Moreover $\vectwo{\tilde{u}}{\tilde{v}}
\in \; span \{ e^{s_1x} {\bf a_2}, e^{s_2 x} {\bf a_1}, e^{s_3 x} {\bf a_1}, e^{s_4 x} {\bf a_2} \}$.
For a solution of (\ref{linearized}) decaying to $(0,0)$ as $x \to \infty$, $\vectwo{\tilde{u}}{\tilde{v}} =
C_1 e^{s_1 x} {\bf a_2} +C_2 e^{s_2 x} {\bf a_1}$ with $C_1$, $C_2$
being constant. 

Recall from (\ref{r21}) that $r_1, r_2$ are the roots of $c_0^2 r^2 + c_0^2 r -\gamma$. This together with 
 $\lambda_1 > \gamma$ yields 
\begin{equation} \label{compare_decay}
s_1<s_2<r_1<-1\;.
\end{equation}
\\

{
Let us remark that $g_1$ is the difference in height between the graph of $f$ and the tangent line $T_1$. 
As $g_1(u_0) \geq 0$ if $u_0 \in {\cal A}_f$, 
in the next section 
certain useful estimates will follow from 
the application of the maximum principle to \eqref{main1}. }
\\

{Next we look at the linearization of (\ref{mainc}) at $(\mu_3,\mu_3/\gamma)$. 
Let $U \equiv u_0-\mu_3$ and $V \equiv v_0- \frac{\mu_3}{\gamma}$. 
Then 
\begin{eqnarray} \label{mainUV}
\left\{  \begin{array}{rl}
\displaystyle dc_0^2 U_{xx}+dc_0^2 U_x+f'(\mu_3) U -V & \ds  =\frac{\mu_3}{\gamma}+f'(\mu_3) (u_0-\mu_3)-f(u_0), \\ \\
\displaystyle c_0^2 V_{xx}+c_0^2 V_x+U-\gamma V &=0.
\end{array} \right.
\end{eqnarray}
{Define $h(u) \equiv \frac{\mu_3}{\gamma}+f'(\mu_3) (u-\mu_3)$ and $g_2(u) \equiv h(u)-f(u)$. 
Since \eqref{t2a} implies 
$g_2(u_0) \geq 0$, we 
apply the maximum principle to } 



\begin{equation} \label{mainUV1}
c_0^2 \vectwo{{U}}{{V}}_{xx}+c_0^2 \vectwo{U}{V}_x- \hat A \vectwo{{U}}{{V}}=  \frac{1}{d} \vectwo{g_2(u_0)}{0}
\geq {\bf 0},
\end{equation}
where
\[
\hat A=\mattwo{\frac{-f'(\mu_3)}{d}}{\frac{1}{d}}{-1}{\gamma} \;.
\]
{Let 
$\hat \lambda_2 > \hat \lambda_1$ be the eigenvalues of $\hat A$. For each eigenvalue $\hat \lambda_i$, denoted by
${\bf \hat a}_i$ an associated eigenvector for $\hat A$ and ${\bf \hat l}_i=(1,\hat \eta_i)^T$ its left eigenvectors.
By the same lines of reasoning, {just by swapping $\beta$ for $-f'(\mu_3)$
with $0<-f'(\mu_3)<-f'(\mu_3^*)=\beta$,}
we obtain several results similar to the above.  
For instance, if $\hat s_2,\hat s_3$ are the roots of $c_0^2 s^2+c_0^2 s -\hat \lambda_1=0$  and $\hat s_1,\hat s_4$, those of
$c_0^2 s^2+c_0^2 s -\hat \lambda_2=0$, then 
\begin{equation} \label{4eigen1}
\hat s_1< \hat s_2<-1<0< \hat s_3<\hat s_4.
\end{equation}
Moreover $\lambda_1<\hat{\lambda}_1< \hat{\lambda}_2<\lambda_2$.



\section{
Positivity 
of $v_0$  
} \label{sec_positiveV}
\setcounter{equation}{0}

\vspace{.1in} 
Now two crucial estimates of $v_0$ will be proved in the next lemma, using the information extracted 
 from linearization.  


\begin{lemma} \label{lem_positiveV}
{If ($H1$) {and ($H2$) are }
  satisfied,} then $v_0>0$ and 
 $v_0 <\frac{\mu_3}{\gamma}$ on $(-\infty,\infty)$. 
\end{lemma}

{Although Lemma~\ref{lem_positiveV} essentially follows from an application of the maximum principle, it requires step by step delicate arguments to verify the following 
observation, which seems to be of independent interest. 
}

\begin{lemma} \label{lem_positive_psi}
Assume that ($H2$) is satisfied.
 If $\psi_i=u_0+ \eta_i v_0$ for $i=1,2$, 
then $\psi_i>0$
everywhere. 
\end{lemma}
\noindent
\begin{proof}
We prove $\psi_2>0$ only, since the other is analogue. \\

\noindent
{\em Step 1:}  By (\ref{poincare}) and $u_0,v_0 \in H^1_{ex}$, the rate of decaying 
 is $O(e^{-x/2})$ as $x \to \infty$. From Lemma~\ref{lem_c1} and Lemma~\ref{lem_case3},
 $u_0 \in C^1({\bf R})$
irrespective if there exist intervals on which $u_0=0$ or {$u_0=\mu_3$} or $u_0=\beta_2$; 
away from these intervals 
 $u_0 \in C^{\infty}$ and $(u_0,v_0)$ satisfies (\ref{main1}).
With 
$u_0 \leq \beta_2 \leq 1+\beta/2$, premultiplying (\ref{main1}) 
by ${\bf l}_2^T$ yields 
\begin{equation} \label{psi2}
c_0^2 {\psi''_2}+ c_0^2 {\psi'_2} -\lambda_2 \psi_2 = -\frac{g_1(u_0)}{d} \leq 0 \;
\quad \mbox{if} \;u_0 \ne 0, \;{ \mu_3}\;\mbox{or} \; \beta_2 \;.
\end{equation}
The last inequality {also} follows from a geometrical description; indeed in the $(u,v)$ plane, $v=-\beta u$ is a line tangent to $v=f(u)$
at $(0,0)$. Within the range of $u_0$ the graph of $v=f(u)$ sits above this tangent line, consequently $g_1(u_0) \geq 0$. Also, we know that 
$v_0 \in C^3({\bf R})$ and  
$\psi_2 \in C^1({\bf R})$.
.
\\ 

\noindent
{\em Step 2:} Since $u_0 \to 0$ and $v_0 \to 0$
as $x \to \infty$, it follows that $\psi_2 \to 0$ as $x \to \infty$.
 Suppose that $\psi_2<0$ somewhere. 
Define $b \equiv \sup \{ x: \psi_2(x)<0 \}$,  where the possibility that $b=\infty$ is not excluded.
Since $\psi_2(b)=0$ (or $\psi_2 \to 0$, if $b=\infty$), in either case there exists a $b_1\in(-\infty,b)$ such that
$\psi_2(b_1) \equiv -t_0 <0$ and $\psi_2'(b_1)\equiv t_1 >0$. 

{Next we claim  $\psi'_2(x) \geq t_1$ for $x \in (-\infty,b_1]$. 
This will be verified separately in different situations: 
} \\

\noindent
Case (A1):  Suppose that $u_0$ has no constrained 
 point 
  in an interval 
$[a_1,b_1]$ for some $a_1\in(-\infty,b_1)$.
With $\psi_2$ satisfying (\ref{psi2}) on $[a_1,b_1]$, it follows from the maximum principle that
$\psi_2$ cannot have a negative minimum in $(a_1,b_1)$. Furthermore, as a consequence of the Hopf lemma,
$\psi_2'>0$ on $[a_1,b_1)$. With such an information putting
back into (\ref{psi2}) gives, for all $x \in [a_1,b_1)$, $\psi''_{2}(x) \leq \lambda_2 \psi_{2}(x)/c_0^2 \leq - \lambda_2 t_0/c_0^2$ and consequently
$\psi'_{2}(x) \geq t_1+ \lambda_2 t_0(b_1-x)/c_0^2  \,{\geq t_1}$.\\

\noindent
Case (A2): Suppose that $u_0 =\beta_2$ on $[a_1,b_1]$ for some $a_1\in(-\infty,b_1)$. 
Since $u_0(b_1)=\beta_2$ and $u_0'(b_1)=0$, 
we infer that $v_0(b_1)= -(t_0+\beta_2)/\eta_2$
and $v_0'(b_1)=t_1/\eta_2 >0$. In view of
\begin{equation} \label{v0eqn}
c_0^2 v_0''+c_0^2 v_0'-\gamma v_0=-\beta_2 \leq 0 \quad \mbox{on}  \;[a_1,b_1] \;,
\end{equation}
the maximum principle
together with the Hopf lemma implies that $v_0(x)<v_0(b_1) = -(t_0+\beta_2)/\eta_2$ and $v_0'(x)>0$ for $x\in[a_1,b_1)$.  
Then $c_0^2 v_0'' \leq -\{ \beta_2+ \frac{\gamma}{\eta_2}(t_0+\beta_2)\}$ and consequently  $v_0'(x) \geq \frac{t_1}{\eta_2}+ \{ \beta_2+ \frac{\gamma}{\eta_2}(t_0+\beta_2)\}(b_1-x)/c_0^2$ for $x \in [a_1,b_1)$.  
Combining with 
 $u_0'(x)=0$ yields 
$\psi'_2(x) \geq t_1+ \{ \eta_2 \beta_2+ \gamma(t_0+\beta_2)\}(b_1-x)/c_0^2 \,{\geq t_1}$ and $\psi_2(x) \leq- t_0$ for $x \in [a_1,b_1]$. Also,
in case $u_0=0$ or {$u_0=\mu_3$} on $[a_1,b_1]$,  
simply replacing $\beta_2$ by $0$  or {$\mu_3$} in the above calculation leads to the same conclusion.
\\


\noindent
Case (A3): Next recall the Case (P3) 
 as indicated just before Lemma \ref{lem_case3}. 
 Suppose 
  as $x \to x_0^+$ 
with 
Case (A1) or (A2) occurs alternately in adjacent subintervals of {$(x_0,b_1]$}, 
then $\psi_2 \in C^1$ and the above argument shows $\psi_2'(x_0) \geq t_1$. However in view of Lemma~\ref{lem_case3}, 
$u_0'(x_0)=v_0'(x_0)=0$, so does 
$\psi_2'(x_0)=0$. This gives rise to a contradiction. 
Similary the case of $x \to x_0^-$ cannot occur either. 

It remains to treat the case where (A1) and (A2) 
 appear alternately 
in consecutive 
intervals 
with no finite limit point
in $(-\infty,b_1]$. In this case, by step 2, $\psi_2'(x) \geq t_1$ for $x\in(-\infty,b_1]$ and thus 
$\psi_2 \to -\infty$ as $x \to -\infty$. This is not possible 
{as both $u_0$ and $v_0$ are bounded}. As a conclusion, $\psi_2 \geq 0$
on $(-\infty,\infty)$;  however 
at this moment we still cannot rule out the possibility of (P3) 
  yet 
  {as $\psi_2$ is non-negative now.} 
\\

\noindent
{\em Step 3:}
Next we show that $\psi_2>0$ on the real line by arguing indirectly. Suppose $\psi_2=0$ at $\xi$, 
then $\psi'_2(\xi)=0$. 
We may assume, without loss of generality, that $\psi_2>0$
on $(\xi,\xi+\delta]$ for some $\delta>0$. 
(For instance, pick $\xi={a_2}$ if there exist $a_1<b_1 \leq a_2<b_2\leq a_3<b_3 \dots$ in the interval $[x_0-\ell,x_0]$ 
such that 
\[ \left\{ \begin{array}{l}
\psi_2>0 \quad \mbox{on intervals} \; (a_i,b_i), \; i=1,2, \dots \;, \\
\psi_2=0 \quad \mbox{on}\; [x_0-\ell,x_0] \setminus \cup_{n=1}^{\infty} (a_n,b_n) \;,
\end{array} \right.
\]
with both $a_n \to x_0^-$ and $b_n \to x_0^-$ as $n \to \infty$.) 
\\

\noindent
Case (B1): Suppose that $u_0(\xi) \ne \beta_2$, {$u_0(\xi) \ne \mu_3$} {and $u_0(\xi) \ne 0$.}

Applying the Hopf lemma to (\ref{psi2}) 
yields $\psi_2'(\xi) > 0$, thus it cannot happen. \\ 

\noindent
Case (B2): {Suppose that $u_0(\xi)=\beta_2$.}

With $u_0$ having a maximum at $\xi$, it follows that $u_0'(\xi)=0$. This together with the corresponding
information on $\psi_2$ implies that
$v_0(\xi) = -\beta_2/\eta_2<0$ and $v_0'(\xi)=0$. Invoking 
(\ref{v0eqn}) yields 
$v_0''(\xi)<-\beta_2/c_0^2<0$. Hence if $x>\xi$ and is sufficiently close to $\xi$ then $v_0(x)<v_0(\xi)=-\beta_2/\eta_2$, and consequently $\psi_2(x) < \psi_2(\xi)=0$. This is absurd, 
so 
(B2) is ruled out. \\ 

\noindent
Case (B3): Suppose that $u_0(\xi)=0$.


{First we consider the case $u_0'(\xi) \ne 0$. 
Since $u_0 \ne 0$ on $(\xi,\xi+\delta]$ for some $\delta>0$,  
$u_0 \in C^2[\xi,\xi+\delta]$ and satisfies (\ref{mainc}a).  
Consequently $\psi_2 \in C^2[\xi,\xi+\delta]$ on the interval $[\xi,\xi+\delta]$, and it is known that $\psi_2$ attains minimum at $x=\xi$
in this interval. 
Applying the Hopf lemma to (\ref{psi2}) yields $\psi_2'(\xi)>0$, which gives rise to a contradiction. 
} \\

{Next in the case of 
$u_0'(\xi)=0$, $v_0(\xi)=v_0'(\xi)=0$ must hold. 
If $u_0 \ne 0$ on $(\xi,\xi+\delta]$, the same proof as above will do. Thus it remains to treat 
 the situation 
when there exist {$b_1>a_1 \geq b_2>a_2\geq b_3>a_3 \dots$} in the interval $[\xi,\xi+\delta]$ 
such that 
\[ \left\{ \begin{array}{l}
u_0 \ne 0 \quad \mbox{on intervals} \; (a_i,b_i), \; i=1,2, \dots \;, \\
u_0=0 \quad \mbox{on}\; [x_0-\ell,x_0] \setminus \cup_{n=1}^{\infty} (a_n,b_n) \;,
\end{array} \right.
\]
with both $a_n \to \xi^+$ and $b_n \to \xi^+$ as $n \to \infty$.
Recall that $u_0$ is in the class $+/-$, so there is $j \geq 1$ such that $u_0$ does not change sign 
on $[\xi,b_j 
]$.
Suppose that $u_0 \geq 0$ on $[\xi,b_j]$. Let us consider the solution of 
(\ref{mainc}b) 
under the initial conditions $v(\xi)=v'(\xi)=0$. Applying a comparison theorem for the initial value problem (Theorem 13, p.26, \cite{PW}), we see that
 $v_0 \leq 0$ on $[\xi,b_j]$, by taking $v=0$ 
 as a comparison function. 
 Then 
 $dc_0^2 u_0''+ d c_0^2 u_0'-(1-u_0)(\beta-u_0) u_0
 =v_0 \leq 0$ on $(a_j,b_j)$. Invoking Hopf lemma yields $u_0'(a_j) > 0$, which is contrary to the corner lemma. }
 {The case of $u_0 \leq 0$ on $[\xi,b_j]$ can be treated similarly.} Now it is clear that 
(B3) cannot occur either.\\

 \noindent
{ Case (B4): Suppose $u_0(\xi)=\mu_3$.}

 The proof is similar to that of 
  (B3) with only slight modification in the case 
 $u_0'(\xi)=0$. The function 
 $u_0-\mu_3$ is in the class $+/-$, 
 it does not change sign on $[\xi,b_j]$. 
  Moreover, for large $j$, $v_0-f(u_0)<0$ in 
   $[\xi, b_j]$, since we know that $v_0'(\xi)=0$ and $v_0(\xi)=- \mu_3/\eta_2<0$. 
Suppose $u_0-\mu_3 > 0$ on $(a_j,b_j)$, then $d c_0^2 u_0''+d c_0^2 u_0'=v_0-f(u_0)<0$ and 
$u_0$ attains its minimum at $a_j$, which implies $u_0'(a_j)>0$ by Hopf lemma. This gives rise to a contradiction. 
   {Next, $u_0-\mu_3 < 0$ on $(a_j,b_j)$ cannot occur either; for otherwise 
    $u_0(a_j)=u_0(b_j)=\mu_3$ 
   would imply that $u_0$ has a minimum 
    in $(a_j,b_j)$, but 
 $dc_0^2 u_0''+dc_0^2 u_0'<0$ on this interval.}
  
  Now the proof is complete. 
%
\end{proof}

{Recall that $U = u_0-\mu_3$ and $V = v_0- \frac{\mu_3}{\gamma}$. A result parallel to Lemma~\ref{lem_positive_psi} is the following.

\begin{lemma} \label{lem_Psi_above}
Assume that ($H1$) 
is satisfied. 
If $\Psi_i=U+ \hat \eta_i V$ for $i=1,2$, 
then $\Psi_i <0$
everywhere. 
\end{lemma}

\begin{proof}
Premultiplying  (\ref{mainUV1}) by $\hat {\bf l}_2^T$ yields 
\begin{equation} \label{Psi2}
c_0^2 {\Psi''_2}+ c_0^2 {\Psi'_2} -\hat \lambda_2 \Psi_2 \geq 0. \;
\end{equation}
It is clear that $\Psi_2(x) \to -\mu_3 -\frac{\mu_3 \hat \eta_2 }{\gamma}<0$ as $x \to \infty$. 

We claim that $\Psi_2 \leq 0$ on $(-\infty,\infty)$; for otherwise,
there is a point $b$ at which $\Psi_2(b)>0$. By the mean value theorem,
there is a point $b_1 \in (b,\infty)$ such that $\Psi_2(b_1) \equiv t_0>0$ and $\Psi_2'(b_1) \equiv - t_1 <0$. Then arguing like 
 Lemma \ref{lem_positive_psi} 
  enables us to 
  conclude that $\Psi_2 \leq 0$
on $(-\infty,\infty)$. 

To show $\Psi_2 < 0$ on $(-\infty,\infty)$, we argue indirectly. Suppose that $\Psi_2(x_m)=0$ for some $x_m$. Then
 there is {an $x_1 \geq x_m$} such that 
$\Psi_2'(x_1)=0$ and $\Psi_2<0$ on $(x_1,x_1+a]$ for some $a>0$.
This violates the consequence from 
applying Hopf lemma to (\ref{Psi2}); 
so 
$\Psi_2<0$ on $(-\infty,\infty)$.
\end{proof}

{We now recall the Case (P3) stated 
 before 
  Lemma \ref{lem_case3}. {Lemma~\ref{lem_positiveV}} 
   rules out 
   the possibility of (P3),  since $v_0 > 0$ is incompatible with Lemma \ref{lem_case3}(a),(b), while Lemma \ref{lem_case3}(c) is equivalent to $V(x_0)=0$. 
} \\

\noindent 
{{\bf Proof of Lemma~\ref{lem_positiveV}.}}
We prove $v_0>0$ only, {since the proof of $V<0$, is analogous.}
We may assume that
 $u_0<0$ somewhere, for otherwise 
the positivity of $v_0$ is an immediate consequence of $(\ref{v_Nsoln})$.
Also, at the points
where $u_0 \leq 0$, 
 the fact that $v_0>0$ simply follows from 
 $\psi_2=u_0+\eta_2 v_0>0$. 
Thus our attention will be 
focusd on the points where $u_0 > 0$.

Suppose that $u_0>0$ on an interval $(a,b)$ with $u_0(a)=u_0(b)=0$. 
The positivity of $\psi_2$ implies that  $v_0(a)>0$ and $v_0(b)>0$.
In view of $c_0^2 v_0''+c_0^2 v_0'-\gamma v_0 =-u_0 < 0$ on $(a,b)$, 
the maximum principle implies that $v_0$ cannot have an interior non-positive minimum. 
The same argument 
 works in case $u_0>0$ on $[a,\infty)$ or $(-\infty, b]$. Now the proof is complete. 

\begin{remark} \label{r_admissible}
Even though $v_0>0$ everywhere, it may not work to add a constraint $v_0 \geq 0$ in the admissible set ${\cal A}_f$. Indeed, {$v_0$ decays like $e^{s_2x}$ for $x$ near $+\infty$,}
while if $w \in C^{\infty}_0({\bf R})$, 
  the decay rate of ${\cal L}_c w$ has to be $e^{r_1x}$. 
 In view of (\ref{compare_decay}), 
  a small perturbation $w$ could result in 
 ${\cal L}_c (u_0+w)$ being negative for $x$ near $+\infty$. 
\end{remark}

\section{Proof of Theorem~\ref{mainThm}: the case $\beta_1 \geq \mu_3$} 
 \label{sec_behavior}
\setcounter{equation}{0}
Invoking the positivity of $v_0$, 
we conclude from the next lemma 
 that the constraint $u_0 \leq \beta_2$ in the admissible set is inactive. 

\begin{lemma} \label{lem_removeBeta2}
If ($H2$) is satisfied, 
 then $u_0<1$. 
  \end{lemma}
 
 \begin{proof}
 {\em Step 1:} 
 First we claim that $u_0<\beta_2$ on the whole real line. Suppose that
 $u_0=\beta_2$ at $x=x_0$. Since {Case {(P3)} for $\nu=\beta_2$} has been ruled out from 
Lemma~\ref{lem_positiveV}, we may assume,
without loss of generality, that $u_0<\beta_2$ on $(x_0,x_0+\delta]$ for some
$\delta>0$.
 Then $u_0$ is of $C^2$ and
 satisfies  $c_0^2 u_0''+c_0^2 u_0'=(v_0-f(u_0))/d \geq \min_{[x_0,x_0+\delta]} v_0 >0$ if $\delta$ is taken so small that $u_0>1$ on 
 $[x_0,x_0+\delta]$ holds. Thus
 $u_0$ cannot have a local interior maximum in $(x_0,x_0+\delta)$. Moreover 
  it follows from Hopf lemma that 
 $u_0'(x_0)<0$, which contradicts the corner lemma.
 The claim is justified {and the constraint $u_0 \leq \beta_2$ is inactive.} \\
 
 \noindent
 {\em Step 2:}
  Suppose there exists
 a local maximum at a point $x$ and 
 $u_0(x) >1$, 
 then
 $c_0^2 u_0''(x)+c_0^2 u_0'(x)+f(u_0(x))-v_0(x)<0$. This is absurd, so if $u(x_1)>1$, 
 then there is an $x_2 > x_1$ such that 
$u_0>1$ and $u_0' \leq 0$ on 
 $(-\infty,x_2)$, and $u_0 \leq 1$ on $[x_2,\infty)$. \\ 

%
%
  
  \noindent
  {\em Step 3:} We next claim that $u_0 \leq 1$ on the whole real line. For otherwise, 
 there is an $\hat x_1 \in (-\infty,x_2)$ such that
  $u_0'(\hat x_1)<0$ and $u_0(\hat x_1)>1$.  Since
  \begin{eqnarray*}
  c_0^2 u_0'' &=& \frac{1}{d}(v_0-f(u_0))-c_0^2 u_0' \\
  & \geq & 0 \qquad \mbox{on} \;\; (-\infty,x_2),
  \end{eqnarray*}
it follows that $u_0'(x) \leq u_0'(\hat x_1)<0$ for
 all $x \in (-\infty,\hat x_1]$. Consequently $u_0 \to \infty$ as $x \to -\infty$, which violates the fact of $u_0$ being bounded. \\
  
  \noindent
  {\em Step 4:} Finally we show that $u_0<1$ on the whole real line by an indirect argument. Suppose $u_0(x_3)=1$ then $u_0'(x_3)=0$ and $x_3$ can be chosen with the property that $u_0<1$ on $(x_3,x_3+\delta]$. By taking 
 $\delta$ small, 
  the function 
$h(x) \equiv u_0(u_0-\beta)>0$ and 
\[
c_0^2 (u_0-1)''+c_0^2 (u_0-1)'-\frac{h(x)}{d} (u_0-1)=\frac{v_0}{d}>0 \quad
\mbox{on} \;\; [x_3,x_3+\delta].\;
\]
Since $u_0-1$ attains a 
{nonnegative maximum at $x=x_3$, the Hopf lemma} requires
that $u_0'=(u_0-1)'<0$ at $x_3$. This gives a contradiction, which completes the proof. 
 \end{proof}

\begin{remark} \label{remark_8.1}
Since the constraint $u_0 \leq \beta_2$ imposed in
 the admissible set  ${\cal A}_f$ is inactive, 
 $u_0 \in C^{\infty}$
 and satisfies (\ref{mainc}a) at all the points except for
 the intervals on which $u_0$ is identical to $0$ or {$\mu_3$}.
\end{remark}

 \begin{lemma} \label{lem_maxBelowBeta}
Assume that ($H2$) is satisfied. If  $ u_0(x_0) \in (0,\beta]$ then $u_0(x_0)$ is not a local maximum.
{The same is true when $u(x_0)=0$, provided that (\ref{mainc}a) is satisfied on $[x_0,x_0+\delta]$
or $[x_0-\delta,x_0]$ for some $\delta>0$.}
 \end{lemma}
 \begin{proof} {Suppose that $0 < u_0(x_0) \leq \beta$ and $u_0(x_0)$ is 
 a local maximum. 
Since $u_0'(x_0)=0$, a direct calculation gives 
$u_0''(x_0)=- u_0'(x_0)-\frac{f(u_0(x_0))}{c_0^2 d}+\frac{v_0(x_0)}{c_0^2 d}>0$, which is absurd.}

{The extension to cover the case of $u(x_0)=0$ is clear.}
 \end{proof}
 
  \begin{lemma} \label{lem_maxBelowmu3}
Assume that ($H1$) is satisfied. 
If $f(u_0(x_0))>f(\mu_3)$ then $u_0(x_0)$ is not a local minimum.
  {The same is true when $u(x_0)=\mu_3$, provided that (\ref{mainc}a) is satisfied on $[x_0,x_0+\delta]$
or $[x_0-\delta,x_0]$ for some $\delta>0$.}
 \end{lemma}
\begin{proof}
Recall that $U=u_0-\mu_3$, $V=v_0-\frac{\mu_3}{\gamma}$ and
 \begin{eqnarray} \label{notlocalmin}
dc_0^2 U_{xx}+dc_0^2 U_x
 -V & \ds  = f(\mu_3) -f(u_0). 
\end{eqnarray}
Then the proof is parallel to that of Lemma~\ref{lem_maxBelowBeta}, using 
 $V<0$ from Lemma~\ref{lem_positiveV}.
\end{proof}
By Lemma~\ref{lem_b},
there is an {$\hat{x}_{\beta}>0$ such that ${u_0 \leq}\; \beta$ if $x\in[\hat{x}_{\beta},\infty)$.} 

\begin{lemma} \label{lem_zeroU} 
$u_0$ cannot be identically zero on any subinterval of 
 {$[\hat{x}_{\beta},\infty)$.}
\end{lemma}

The proof is same as that of 
 Lemma 6.4 of \cite{CC1}. We omit it.

\vspace{.2in}
As a consequence of Lemma~\ref{lem_zeroU}, 
$(u_0,v_0)$ satisfies (\ref{mainc}) on $[\hat{x}_{\beta},\infty)$. 
Following the discussion after (\ref{4eigen}),
there exist $C_1,C_2 \in \bf R$ such that 
$\vectwo{u_0}{v_0} \sim C_1 e^{s_1 x} {\bf a_2} +C_2 e^{s_2 x} {\bf a_1}$ as $x \to \infty$. 
It will be referred to as 
 slow decay at $+\infty$ if $C_2 \ne 0$, otherwise 
  fast decay.

\begin{lemma} \label{lem_slowDecay}
Assume that ($H2$) is satisfied. Then as $x \to \infty$, 
\[
\vectwo{u_0}{v_0} \sim 
C_2 e^{s_2 x} {\bf a}_1
\] 
 with $C_2<0$, and $u_0 <0$ for sufficiently large $x$.
\end{lemma}
\begin{proof} 
Suppose that 
 $\vectwo{u_0}{v_0} \sim C_1 e^{s_1 x} {\bf a_2}$ near $+\infty$ with $C_1 \ne 0$. 
 Then $\psi_2 ={\bf l}_2 \cdot \vectwo{u_0}{v_0} \sim C_1 e^{s_1x} {\bf l}_2 \cdot {\bf a}_2$. Note that $\psi_2>0$ by Lemma~\ref{lem_positive_psi} and ${\bf l}_2 \cdot {\bf a}_2>0$ by (\ref{partd}), this implies 
  $C_1>0$. 
  From the sign of each component of ${\bf a}_2$, $C_1>0$ implies that $u_0(x)>0$ and $v_0(x)<0$ for $x$ near $+\infty$. 
But this is contrary to $v_0 > 0$. Hence 
 $C_2 \ne 0$. \\ 
 
 Taking inner product with ${\bf l}_1$ yields $\psi_1 \sim C_2 e^{s_2x} {\bf l}_1 \cdot {\bf a_1}$. Then (\ref{partd}) and 
  $\psi_1 > 0$ imply $C_2<0$, from which the last assertion follows.  
\end{proof} 

Let ${\bf \hat a}_1=(\hat \eta_2,-1)^T$ and ${\bf \hat a}_2=(\hat \eta_1,-1)^T$. If $(u_0,v_0) \to (\mu_3,\mu_3/\gamma)$
as $x \to -\infty$, its asymptotic behavior 
 can be obtained in the same manner. 
\begin{lemma} \label{lem_slowDecayB}
Assume that ($H1$) is satisfied. Suppose $(u_0,v_0) \to (\mu_3,\mu_3/\gamma)$
as $x \to -\infty$ {and $u_0$ is not identical to $\mu_3$ on any subinterval,} then 
\[
\vectwo{u_0}{v_0} \sim \vectwo{\mu_3}{\mu_3/\gamma} + \hat C_1 {e^{\hat s_4 x}} {\bf \hat a_2} + \hat C_2 {e^{\hat s_3 x}} {\bf \hat a_1}
\] 
 with {$\hat C_2>0$}. 
\end{lemma}
}
%
Recall that $\sup u_0>\beta_1$ and $u_0<0$ for $x$ being sufficiently large. 
Define {$\zeta_0 \equiv \sup \{ x: u_0(x) > 0 \}$}. 
It is clear that ${\hat{x}_{\beta}} < \zeta_0 < \infty$.

  \begin{lemma} \label{lem_u_negative} 
Assume that ($H2$) is satisfied.  
 
 \noindent
 (a) $u_0'(\zeta_0)<0$, $u_0 \geq 0$ on $(-\infty,\zeta_0]$ and 
 $u_0 <0$ on $(\zeta_0,\infty)$. 
 
  \noindent
 (b) $u_0$ has a unique 
minimum at $\zeta_m$
in $(\zeta_0,\infty)$ which is the global minimum of $u_0$. 
{ Moreover $u_0'<0$ on $[\zeta_0,\zeta_m)$ and $u_0'>0$ on $(\zeta_m,\infty)$.}

 \noindent
 (c) $v_0'<0$ on $[\zeta_0,\infty)$.
 \end{lemma}
 
 \begin{proof}
Since
 $d c_0^2 u_0''+d c_0^2 u_0'+f(u_0)=v_0 > 0$ on $[\zeta_0,\zeta_0+\delta]$ and $f(u_0)/u_0=-(1-u_0) (\beta-u_0)<0$
because $\zeta_0>\hat{x}_{\beta}$, applying the Hopf lemma gives
$u_0'(\zeta_0)<0$. 
 For $u_0$ to be in the class $+/-$, it {forces} $u_0 \leq 0$ on 
$(\zeta_0,\infty]$. Furthermore {the proof of} Lemma~\ref{lem_zeroU} shows that $u_0$ cannot be vanished 
on a subinterval of $[\zeta_0,\infty)$;  
{while Lemma~\ref{lem_maxBelowBeta} 
does not allow that $u_0(x)=0$ at a point $x \in(\zeta_0,\infty)$.}
The proof of (a) is complete.
%

We next prove (c). 
Note that $d c_0^2 v_0''+d c_0^2 v_0'=\gamma v_0 -u_0 >0$ implies that
 $v_0$ cannot have an interior maximum on the interval $[\zeta_0,\infty)$. 
For $x$ being sufficiently large, it is already known that 
$v_0'<0$.
Invoking the Hopf lemma,
we infer that $v_0'<0$ on $[\zeta_0,\infty)$.\\

 Differentiating (\ref{mainc}a) gives 
\begin{equation} \label{uPrime}
d c_0^2 u_0'''+d c_0^2 u_0''+f'(u_0) u_0'=v_0'<0
\end{equation}
on $(\zeta_0,\infty)$. Since $f'(u_0)<0$ for $u_0<0$, the maximum principle asserts that 
$u_0'$ cannot have a non-positive minimum on $(\zeta_0,\infty)$. Together with the Hopf lemma,
we conclude that there is a $\zeta_m \in (\zeta_0,\infty)$ such that $u_0'<0$ on $[\zeta_0,\zeta_m)$,
 $u_0' >0$ on $(\zeta_m,\infty)$ {and $u_0''(\zeta_m)>0$}. Thus $u_0$ has a unique negative local minimum 
at $\zeta_m$. This is also a global minimum of $u_0$ as it 
 is in the
class $+/-$. The proof of (b) is complete, so is the lemma. 
 \end{proof}


%
%
%

 \begin{lemma} \label{lem_zeroInterval}
 Suppose $u_0(x_1) > 0$ and $u_0(x_2) > 0$, then $u_0$ cannot be identically zero 
  on any interval contained in $(x_1,x_2)$.

 \end{lemma}
 
\begin{proof} 
{As $u_0$ is in the class $+/-$, it is necessary that $u_0 \geq 0$ on $(x_1,x_2)$.}
We argue indirectly by assuming that 
 $u_0 \equiv 0$ on some interval $[a,b] \subset (x_1,x_2)$ with {$b>a$}. 
Let 
\begin{equation} \label{decompose}
u_1(x)= \left\{ \begin{array}{ll}
u_0(x) &
\mbox{if} \; x \leq a \\ \\
0 &
\mbox{if} \; x>a
\end{array} \right.
\end{equation}
and $u_2=u_0-u_1$. 
{Observe that $u_1$ is in the class $+$; $u_1 - \mu_3$ is in the class $+/-$;
 $u_2$ is in the class $+/-$ and $u_2 -\mu_3$ is in the class $-$.} Then 
 similar to the proof of Lemma 7.2 (b) of \cite{CC1}, we get a contradiction which completes the proof. 
 \end{proof} 
 
%

\begin{remark} \label{remark_8.3}
\noindent
(a)  As of now it is possible  that $u_0=0$ at isolated points in $(x_1,x_2)$ in the above lemma. However 
$(u_0,v_0)$ satisfies (\ref{mainc}) in neighborhoods of such isolated points as required by
Lemma~\ref{lem_c1}.\\
\noindent
(b) As a consequence, either 
$u_0$ is not identically zero on any subinterval of $(-\infty,\infty)$ or 
$u_0$ is vanished 
 on one interval $(-\infty, x_c)$ only, for some $x_c \in {\bf R}$. 
\end{remark}

Recall that $\beta_1$ is the unique point in $(\beta,1)$ which satisfies $\int_0^{\beta_1}f(\eta) \,d \eta=0$. If $\beta_1 \geq \mu_3$, it is clear from $\sup u_0 > \beta_1$ that the set $\{ x: u_0(x) > \mu_3 \}$ is nonempty. 
Define $\zeta_\mu \equiv \sup \{ x: u_0(x) {> \mu_3} \}$. 

\begin{lemma} \label{lem_remove_mu3}
 If $\{ x: u_0(x) > \mu_3 \}$ is nonempty,
then
$(u_0,v_0)$ is a traveling wave solution {with $u_0,v_0 \in C^{\infty}({\bf R})$.}
Moreover $u_0-\mu_3>0$ on $(-\infty,\zeta_\mu)$.

\end{lemma}
\begin{proof}
Recall that $U=u_0-\mu_3$ { lies} in the class $+/-$. 
Suppose $u_0$ is identical to $\mu_3$ 
 on a subinterval $[a,b]$ of $(-\infty,\zeta_\mu)$, we may assume
that $u_0>\mu_3$ on $(b,b+\delta]$ for some small $\delta>0$. In addition $u_0'(b)=0$ by the corner lemma.
Also, it is clear that $u_0''(b^+) \geq 0$. {Then evaluating} at $x=b^+$ yields a contradiction as follows: 
\[
0=\lim_{x \to b^+}c_0^2 u_0''(x)+c_0^2 u_0'(x)+f(u_0(x))-v_0(x) \geq f(\mu_3)-v_0(b^+)>0, 
\]
 by making use of Lemma~\ref{lem_positiveV}.

{The same argument yields 
 a contradiction if $u_0$ is identical to $\mu_3$ 
  on $[a,\zeta_\mu]$ by evaluating at the point $x=a^-$,} so does for the case 
that $u_0(x_0)=\mu_3$ at an isolated point $x_0 \in (-\infty,\zeta_\mu)$.
Thus $u_0-\mu_3>0$ on $(-\infty,\zeta_\mu)$. Now it confirms that $(c_0,u_0,v_0)$ is a traveling wave solution. 
\end{proof}

{The proof of $\lim_{x\to -\infty} (u_0(x),v_0(x))=(\mu_3,\mu_3/\gamma)$ will be given in the next section after we treat the case $\beta_1 < \mu_3$.}





\section{
The case $\beta_1 < \mu_3$} 
 \label{sec_complete}
\setcounter{equation}{0}

As to show Theorem~\ref{mainThm} for the case $\beta_1 < \mu_3$, 
it needs more analysis to prove that
$\{ x: u_0(x) > \mu_3 \}$  is non-empty. 
{As a remark, unless pointed out explicitly all the lemmas in this section are valid irrespective of the order of $\beta_1$ and $\mu_3$.}

The inequalities stated in (\ref{poincare}) and (\ref{v_bound}) are useful for studying the asymptotic behavior of $(u_0,v_0)$ for $x$ near 
$+\infty$, while understanding the 
behavior of a traveling wave 
near $-\infty$ {requires} more efforts  
 when such a solution is obtained from the weighted function space $H^1_{ex}$ 
via a variational approach.

 By 
 Corollary \ref{cor_zeroJ}
there is a $c_0 \in [\underline{c},\bar{c}]$ such that 
${\cal J}_f(c_0)=0$. 
For 
 $d >0$,
define $c_0^*=c_0^*(d)$=$\max {\cal S}_f$, where ${\cal S}_f=\{c: c \in (0,\infty) 
~\mbox{and}~{\cal J}_f(c)=0\}$. If ${\cal S}_f$ is nonempty, 
 $c_0^* \leq \bar{c}$ simply follows from Lemma \ref{lem_largec}. 
Our attention now turns to the fast speed traveling front; that is, the traveling front with speed $c_0^*$. 
For each $d$,  
$c_0^*(d)$ will be simply written as $c_0(d)$ or $c_0$, with $u_0$ being a minimizer of 
{$J_{c_0^*}$} in ${\cal A}_f$. \\ 

Recall from Lemma~\ref{lemCutoff1} that
$d c_0^2$ is bounded from above as $d \to 0$. 
The following investigation is {another step towards understanding how the speed $c_0$ depends on $d$ asymptotically.}

\begin{lemma} \label{lem_noFinite_c}
If (H1) 
is satisfied, 
$c_0(d) \to \infty$ as $d \to 0$.
\end{lemma}

Abbreviated notation will be 
 used to express certain properties with respect to 
  a sequence $d \to 0$; as an example, 
  the above lemma 
 means that $c_0^{(n)} \equiv c_0(d^{(n)}) \to \infty$ along any sequence $d^{(n)} \to 0$.  
Its proof is similar to that of Lemma 8.4 of \cite{CC1}, using Lemma~\ref{lem_test_fun} together with an indirect argument. 
Then 
 arguing like Lemma 8.5 of \cite{CC1} yields 
\begin{eqnarray}
{\lim \inf}_{n \to \infty} \; d^{(n)} (c_0^{(n)})^2  
& \geq  &  d^{(l)} (c_0^{(l)})^2  +{\frac{1}{2}} \int_{\bf R} e^x \, u_0^{(l)} \, {\cal L}_{c_0^{(l)}} u_0^{(l)}\, dx \; \label{finite_c2}
\end{eqnarray}
for any fixed $l \in {\bf N}$. In particular (\ref{finite_c2}) holds for any sequence $d^{(n)} \to 0$, since we may 
add 
 $d^{(l)}$ into the sequence $\{d^{(n)}\}$ to form a new sequence. 
 Setting $\delta_m=d^{(l)} (c_0^{(l)})^2$ yields the following result. 

\begin{lemma} \label{lem_finite_c1}
There exists 
 ${\delta}_m > 0$ such that ${\lim \inf} \, (c_0^{(n)})^2d^{(n)} > {\delta}_m$ holds along any sequence $d^{(n)} \to 0$.  

\end{lemma}

Recall that $\delta_0=  
\frac{(1-2\beta)^2}{2}$ and define
\[
I_*(w) = \int_{-\infty}^\infty e^x \{ \frac{\delta_0}{2} w'^2 + F(w) \}. 
\]
 {Let $a_*=-2 (1-2\beta) \tanh^{-1}(1-2\beta_1)$ and ${{\cal H}_{1 \to 0}}(x)= \frac{1}{2}-\frac{1}{2} \tanh (\frac{x-a_*}{2 (1-2\beta)})$.} It is known that ${{\cal H}_{1 \to 0}}$ is a minimizer of $I_*$, {${\cal H}_{1 \to 0}(0)=\beta_1$} and $\inf_{w \in H^1_{ex}} I_*(w) = 0$.

\begin{lemma} \label{lem_finite_c2}
$d c_0^2 < \delta_0$. 
\end{lemma}
\begin{proof}
Recall that $J_{c_0}(u_0)={\cal J}_f(c_0)=0$. Suppose $d c_0^2 \geq \delta_0$ then 
\begin{eqnarray*}
0 = \inf_{w \in H^1_{ex}} I_*(w) \leq \int_{-\infty}^\infty e^x \{ \frac{\delta_0}{2} u_0^2 + F(u_0) \} \, dx < J_{c_0}(u_0) = 0,
\end{eqnarray*}
which is absurd. 
\end{proof}

{
Set $\zeta_\beta \equiv \sup \{ x: u_0(x)=\beta_1 \}$.} The 
 next lemma gives
 upper and lower bounds for $\zeta_\beta$, both are independent of $d$, 
 as 
  in Lemma 8.6 of \cite{CC1}. 
{That $\inf u_0 \to 0$ follows from $\sup_{[a,\infty)} v_0 \to 0$ and (\ref{mainc}a).}

\begin{lemma} \label{lem_small_v}
Let $z_3^+ = \log (\frac{2}{\beta_1^2})$ and $z_3^- = \log (\frac{12 \delta_m}{(1-2\beta)})$. 
Then $z_3^- \leq \zeta_\beta \leq z_3^+$ and $\inf u_0 \to 0$ 
as $d \to 0$. Moreover for any fixed $a \in {\bf R}$, $\sup_{[a,\infty)} v_0 \to 0$ as $d \to 0$. 
\end{lemma}

 On the  investigation of 
asymptotic behavior, 
  we need a number of estimates. 

\begin{lemma} \label{lem_reflectF}
{If $0 < a < 1-\mu_3$, then $F(\mu_3+a) < F(\mu_3-a)$.}
\end{lemma}

The proof follows from straightforward calculation. We omit it. \\

The next observation is 
$\sup u_0 \approx 1$ for small $d$, which will be used for 
 studying the behavior of $u_0$ near $x=-\infty$. 
To accomplish this task, we employ an argument based on a surgery on the minimizer; however it is not 
 a small perturbation. 
{By Lemma~\ref{lem_finite_c1}  
 and Lemma~\ref{lem_finite_c2}, $d c^2$ is bounded from above and below.  
Let $\tilde{\delta} \in [\delta_m,\delta_0]$, $E(\beta_1) \equiv \{ w \in H^1_{ex}(-\infty,0): w(0)=\beta_1 \}$ and $I = 
I_{\tilde{\delta}}
: E(\beta_1) \to {\bf R}$ defined by 
\[
I(w) \equiv \int_{-\infty}^0 e^x \{ \frac{\tilde{\delta}}{2} w'^2 + F(w) \} \, dx.
\]
An auxiliary function constructed in the next lemma will be used later.
\begin{lemma} \label{lem_w0_nagumo}
There is a unique minimizer $w_0$ of $I$ in $E(\beta_1) 
$. Moreover 
$w_0'<0$ on $(-\infty,0]$ and $w_0 \to 1$ as $x \to -\infty$.
\end{lemma}
\begin{proof}
Since $I(w) \geq \int_{-\infty}^0 e^x F(w) \, dx \geq F_{min} \int_{-\infty}^0 e^x \, dx=- \frac{1-2\beta}{12}$,
$\inf_{w \in {E(\beta_1)}} I(w)$ exists and there is a minimizing sequence $\{ w_n \}_{n=1}^{\infty} \subset {E(\beta_1)}$
such that $I(w_n) 
\to \inf_{w \in {E(\beta_1)}} I(w)$. Then by standard variational argument, we obtain a minimizer 
$w_0$ of $I$. Moreover $w_0 \in C^{\infty}(-\infty,0]$ and it satisfies 
\begin{equation} \label{el_nagumo}
\tilde{\delta} w_0'' +\tilde{\delta} w_0'+f(w_0)=0
\end{equation}
 with $w_0(0)=\beta_1$. It is not difficult to show that  $0 \leq w_0 \leq 1$, since $F(u) \geq F(1)$ for $u \geq 1$ and $F(u) \geq F(0)$ for $u \leq 0$. 
The presence of a damping term $\tilde{\delta} w_0'$ in (\ref{el_nagumo}) together with the Poincare-Bendixson theorem asserts that there is no periodic solution. Then as $x \to -\infty$, $w_0$ goes to one of the three zeros 
 of $f$, namely $\{0, \beta, 1\}$.  
{In fact, 
$\frac{d}{dx} (\frac{\tilde{\delta}}{2} w_0'^2 -F(w_0)) = -\tilde{\delta} w_0'^2$. Since $w_0$ is not a constant function,
it follows that $\frac{\tilde{\delta}}{2} (w_0'(0))^2 - F(\beta_1)+ F(w_0(-\infty))<0$. With $F(\beta_1)=0$, we conclude that
$F(w_0(-\infty))<0$ and $w_0(-\infty)=1$ is the only feasible choice; 
thus $w_0 \to 1$ as $x \to -\infty$.}


We claim that $1 > w_0 \geq \beta_1$ on $(-\infty,0]$. Suppose $0 \leq w_0 < \beta_1$ for some $x$, then setting \[ 
w_{new}(x) \equiv  \left\{ \begin{array}{ll}
w_0(x), & \mbox{if} \; \beta_1 \leq w_0(x) \leq 1, \\
\beta_1, & \mbox{if} \;0 \leq w_0(x)<\beta_1.
\end{array} \right.
\]
gives $I(w_{new}) < I(w_0)$, which is contrary to  
$w_0$ being a minimizer. 
Furthermore $w_0<1$ on $(-\infty,0]$, for otherwise $w_0 \equiv 1$, since  
 {for any $x_1$,} $w_0(x_1)=1$ {forces} $w_0'(x_1)=0$. Clearly the boundary condition
$w_0(0)=\beta_1$ rules out the possibility $w_0 \equiv 1$, we thus complete the proof of $1 > w_0 \geq \beta_1$ on $(-\infty,0]$. 

Next we claim that $w_0$ is strictly decreasing on $(-\infty,0]$. If not, 
 there exists a point $x_1 \in (-\infty,0)$ such that $w_0(x_1)$ is a local maximum. With
 $w_0(x_1)<1$ and $w_0 \to 1$ as $x \to -\infty$,  there is a point $x_0 < x_1$ such that $w_0(x_0)$ is a local minimum. 
This leads to a contradiction, since
 $w_0''+w_0'= -\frac{f(w_0)}{\tilde{\delta}} \leq 0$ on $(-\infty,0]$. 
 Applying the Hopf lemma on $w_0''+w_0' \leq 0$ yields 
$w_0'<0$ on $(-\infty,0]$. {Using a} 
 phase plane analysis for (\ref{el_nagumo}), the minimizer $w_0$ is unique, 
 as an unstable manifold of the saddle point $(w_0,w_0')=(1,0)$.  
\end{proof}

In addition to $w_0$, 
 another auxiliary function can be constructed from 
\[ 
 E(\mu_3) \equiv \{ w \in H^1_{ex}(-\infty,0): w(0)=\mu_3 \}.
 \]

\begin{lemma} \label{lem_W0_nagumo}
There is a unique minimizer $W_0$ of $I$ in $E(\mu_3)$. Moreover 
$W_0'<0$ on $(-\infty,0]$ and $W_0 \to 1$ as $x \to -\infty$.
\end{lemma}
\begin{proof}
Following the proof of Lemma~\ref{lem_w0_nagumo}, we obtain 
a minimizer $W_0 \in E(\mu_3) \cap  C^{\infty}(-\infty,0]$. Moreover 
\begin{equation} \label{EL_nagumo}
\tilde{\delta} W_0'' +\tilde{\delta} W_0'+f(W_0)=0,
\end{equation}
 $W_0(0)=\mu_3$ and $0 \leq W_0 \leq 1$. 
 By Poincare-Bendixson theorem, $W_0$ approaches to an equilibrium; that is,  
 one of $\{0,\beta,1 \}$ 
 as $x \to -\infty$.
 {A simple calculation gives}
$\frac{d}{dx} (\frac{\tilde{\delta}}{2} W_0'^2 -F(W_0)) = -\tilde{\delta} W_0'^2$. Since $W_0$ is not a constant function,
it follows that $\frac{\tilde{\delta}}{2} (W_0'(0))^2 - F(\mu_3)+ F(W_0(-\infty))<0$. Furthermore, $F(\beta)>F(\mu_3)$ implies that 
 $W_0(-\infty)=0$ or $1$. 
 
Recall that $1>\mu_3>\hat{\rho}>1/2$. Suppose $W_0(-\infty)=0$, then we define 
 \[
 W_{new}(x) \equiv \left\{ \begin{array}{ll}
 2 \mu_3 -W_0(x), & \mbox{if} \;\;  2\mu_3-1 \leq W_0(x) \leq \mu_3 \;,\\
 1, & \mbox{if} \;\; 0 \leq W_0(x) < 2 \mu_3 -1 \;, \\
 W_0(x), & \mbox{if} \;\; W_0(x) > \mu_3 \;.
 \end{array} \right.
 \]
Note that $W_{new} \in E(\mu_3)$ and we have a number of 
observations: \\ 
\noindent
(i) For those $x$ such that $W_0(x) \leq 2 \mu_3-1$, 
$F(W_{new}(x))-F(W_0(x))=F(1) -F(W_0(x))<0$; \\
\noindent
(ii) 
If $2 \mu_3-1 < W_0(x) < \mu_3$, 
Lemma~\ref{lem_reflectF} gives 
 $F(W_{new}(x)) \leq F(W_0(x))$;\\
\noindent
(iii) For those $x$ such that $W_0(x) \geq \mu_3$, it is clear that $F(W_{new}(x))=F(W_0(x))$. \\
\noindent
Direct calculation 
yields $I(W_{new})<I(W_0)$, which {contradicts
$W_0$ being} a minimizer. Hence 
 $W_{0} \to 1$ as $x \to -\infty$.
 
 The above argument {also shows} that $W_0 \geq \mu_3$ on $(-\infty,0]$, and  
 the rest of the proof is parallel to that for Lemma~\ref{lem_w0_nagumo}. 
\end{proof}

Recall that $\zeta_\mu \equiv \sup \{ x: u_0(x)=\mu_3\}$ if the set $\{ x: u_0(x) > \mu_3 \}$ is nonempty, and $\zeta_\beta \equiv \sup \{ x: u_0(x)=\beta_1 \}$.
Let us remark that $\zeta_{\beta}$ always stays 
 finite. 
The next lemma shows that for small $d$ the set $\{ x: u_0(x) > \mu_3 \}$ is nonempty even if $\mu_3>\beta_1$. 
As 
 an immediate consequence, we see that the proof of 
Theorem~\ref{mainThm} is complete. 
\begin{lemma} \label{lem_above_mu3}
There is a {$d_f={d_f(\gamma)}>0$ such that $\{ x: u_0(x) > \mu_3 \}$ is nonempty if $d<d_f$}. 
 Moreover $u_0'(\zeta_{\beta})<0$, $u_0'(\zeta_\mu)<0$ and $\sup u_0 \to 1$ as $d \to 0$. 
\end{lemma}
}

Let $w_0$ be the unique minimizer
of $I_{\tilde{\delta}}$ over $E(\beta_1)$. The following lemma will be used to prove Lemma~\ref{lem_above_mu3}.

\begin{lemma} \label{unidiff}
Suppose 
 $u_0(x) \leq \mu_3$ and $dc_0^2=\tilde{\delta}$, 
  then there is an $\hat M> 0$ such that 
%
 \begin{eqnarray} \label{unidiff1}
 \int_{-\infty}^{0} e^x \{ (F(w_0(x))+\frac{\tilde{\delta}}{2} {(w'_0(x))}^{2})
- (F(u_{0}(x+\zeta_\beta))+\frac{dc_0^2}{2} {(u'_0(x
+\zeta_\beta))}^2)\}\, dx < -\hat M 
\end{eqnarray}
holds with 
$\hat M$ not 
depending on $d$ or $\tilde{\delta}$. 
\end{lemma}

\noindent
{Proof of Lemma~\ref{lem_above_mu3}.}
{We argue indirectly.  
Suppose that $\beta_1<\mu_3$ and $\{ x: u_0(x) > \mu_3 \}$ is empty. 
Set 
\[
u_{new}(x)= \left\{ \begin{array}{ll} w_0(x), & \mbox{if} \; x<0 \\ 
                                                      u_0(x+\zeta_\beta), & \mbox{if} \; x \geq 0, 
  \end{array} \right.
\]
and $v_{new} \equiv {\cal L}_{c_0} u_{new}$. Clearly $u_{new}-\mu_3$ is in the class $+/-$. 
Note that there are  upper and lower bounds for $\zeta_\beta$, 
which are independent of $d$, 
so are $\int_{\bf R} e^x {u'}_{new}^2 \,dx$. 
If $dc_0^2=\tilde{\delta}$ then by Lemma~\ref{unidiff}
\begin{eqnarray*}
 J_{c_0}(u_{new})-J_{c_0}(u_0) < -\hat M
+ \int_{-\infty}^{0} \frac{e^x}{2} (u_{new}-u_0) \;(v_{new}+v_0) \, dx. 
\end{eqnarray*}
Furthermore there is an $L>0$ 
 such that
\begin{eqnarray*}
\int_{-\infty}^{-L} \frac{e^x}{2} (u_{new}-u_0) \;(v_{new}+v_0) \, dx 
\leq {\frac{2}{\gamma}}\int_{-\infty}^{-L} \frac{e^x}{2} \; dx 
 \leq \frac{ e^{-L}}{\gamma}  < \hat M/2. 
\end{eqnarray*}
On the interval $[-L,0]$, (\ref{v_H1}) and Lemma~\ref{lem_noFinite_c} imply that 
$|v_0|,|v_{new}| \to 0$ as $d \to 0$. Thus there is a $d_f > 0$ such that if $d<d_f$ then 
\begin{eqnarray*}
\int_{-L}^{0} \frac{e^x}{2} (|v_{new}|+v_0) \; dx < {\hat M/2},  
\end{eqnarray*}
and consequently 
 $J_{c_0}(u_{new})<0$. 
{It follows that} there is an $a \in \bf R$ such that $u_{new}(\cdot-a) \in {\cal A}_f$ but $J_{c_0}(u_{new}(\cdot-a))<0$.
This yields a contradiction, so from Lemma~\ref{lem_remove_mu3} we conclude that $(u_0,v_0)$ is a bounded smooth traveling front solution
 {with $\sup u_0 > \mu_3$.}

{With slight modification, the above argument shows that $u_0(\cdot+\zeta_\beta) \to w_0$ in $C^{\infty}_{loc}(-\infty,0]$ as $d \to 0$. 
Hence if $d$ is small enough then
$u_0$ rises to nearly $1$ once
it touches $\beta_1$ 
at $\zeta_\beta$. 
Moreover with $w_0'<0$, it is necessary that $u_0'(\zeta_{\beta})<0$ and $u_0'(\zeta_{\mu})<0$.}

{A similar conclusion can be drawn if $\beta_1 \geq \mu_3$.
Since $\zeta_{\mu}$ is well defined, 
we simply 
replace $w_0(\cdot-\zeta_{\beta})$ by $W_0(\cdot-\zeta_{\mu})$ in the above construction, as studied in 
Lemma~\ref{lem_W0_nagumo}.} The proof is complete. \hfill\(\Box\)
\\ 





\noindent
{{Proof of Lemma~\ref{unidiff}.}
Suppose that the assertion of the lemma is false. Then there is a sequence $\{d^{(n)}\}$ 
 such that $(c_0^{(n)})^2d^{(n)} \to \tilde{\delta}$ if $n \to \infty$ and along this sequence
 \begin{eqnarray} \label{unidiff2}
 \int_{-\infty}^{0} e^x (F(u_{0}(x+\zeta_\beta))+\frac{dc_0^2}{2} {(u'_0(x
+\zeta_\beta))}^2)\}\, dx \to \inf_{w \in E(\beta_1)} I_{\tilde{\delta}}(w).
\end{eqnarray}
This implies that {a subsequence of $u_0(\cdot+\zeta_\beta)$ converges 
weakly in $H^1_e$ and strongly in $C^2_{loc}$}
to the unique minimizer of $I_{\tilde{\delta}}$, which is absurd because the limit of $u_0(\cdot+\zeta_\beta) \leq \mu_3$. The proof is complete. 
\hfill\(\Box\)
\\
} 

\begin{lemma} \label{lem_uLimit} 
If $\lim_{x\to -\infty}u_0(x)$ exists 
then $u_0 \to \mu_3$ and $v_0 \to \mu_3/\gamma$ as $x \to -\infty$.
The same conclusion holds if $\lim_{x\to -\infty}v_0(x)$ exists. 
\end{lemma}

\begin{proof}
We only prove the first assertion, the second is analogue. 
Suppose $\lim_{x \to -\infty} u_0(x) =r_0$, since $\| u_0'' \|_{L^{\infty}({\bf R})}$ is bounded, we
apply an interpolation theorem to obtain $u_0' \to 0$ as $x \to -\infty$. 
Together with $\|u_0'''\|_{L^{\infty}({\bf R})}$ being bounded, 
another application of interpolation theorem leads to $u_0'' \to 0$ as $x \to -\infty$. 
Then it follows from (\ref{mainc}a) that $v_0 \to s_0 \equiv f(r_0)$ as $x \to -\infty$, which shows 
 that $v_0' \to 0$ and $v_0'' \to 0$ as $x \to -\infty$.
From (\ref{mainc}b),
we see that $\gamma s_0=r_0$. 
\end{proof}

{Now putting Lemmas~\ref{lem_remove_mu3}, {\ref{lem_small_v}} and \ref{lem_above_mu3} together shows that $(c_0,u_0,v_0)$ is a  traveling wave solution. Moreover $(u_0,v_0)$ is of $C^{\infty}$ and $u_0>\mu_3$ on $(-\infty,\zeta_\mu)$. Recall that $U=u_0-\mu_3$, $V=v_0-\mu_3/\gamma$ and  $V<0$ everywhere. 
Since $c_0^2 V'' + c_0^2 V' -\gamma V = -U \leq 0$ on {$(-\infty, \zeta_{\mu}]$}, 
 $V$ cannot have a non-positive minimum. 
  Thus $V$ is monotone
near $x=-\infty$ and  $\lim_{x\to -\infty}V(x)$ exists. 
  Since $u_0>\mu_3$ near $x=-\infty$, it follows from Lemma~\ref{lem_uLimit} that $\lim_{x\to -\infty} (u_0(x),v_0(x))=(\mu_3,\mu_3/\gamma)$, in view of the nullclines. Moreover $v_0$ is monotone decreasing on $(-\infty,\zeta_\mu)$. 
Applying Hopf lemma yields $V'=v_0'<0$ on $(-\infty,\zeta_{\mu}]$. The proof of Theorem~\ref{mainThm} is complete.}  \\ 
A further investigation leads to 
the asymptotic behavior on $c_0$ and $u_0$. We do not give a proof for the following lemma, as it is similar to that of Lemma 8.6 of \cite{CC1}. 


\begin{lemma} \label{lem_finite_c}
If $d \to 0$ then $dc_0^2 \to \delta_0$, $u_0(\cdot+\zeta_\beta) \to {{\cal H}_{1 \to 0}}$ in $C^{\infty}_{loc}({\bf R})$ and $\sup {u_0} \to 
1$.  
\end{lemma}

For the profile of $(u_0,v_0)$, it is 
 already known that there exist $\zeta_\mu < \zeta_0 < \zeta_m$ such that  \\
(i) $u_0(\zeta_\mu)=\mu_3$, $u_0'(\zeta_\mu)<0$, and $u_0>\mu_3$ on $(-\infty,\zeta_\mu)$; as $d \to 0$ the upper and lower bounds of $\zeta_\mu$ are independent of $d$;\\
(ii) $u_0(\zeta_0)=0$, $u_0'(\zeta_0)<0$, $u_0<0$  on $(\zeta_0,\infty)$, and $u_0> 0$ on $(-\infty,\zeta_0)$;\\
(iii) $u_0$ attains a negative minimum at $\zeta_m$.

\vspace{.1in}
Since $\lim_{x\to -\infty} (u_0(x),v_0(x))=(\mu_3,\mu_3/\gamma)$, $u_0$ has at least 
 one positive maximum 
in $(-\infty,\zeta_0)$. 
Define
\[
\zeta_M \equiv \sup \{x:  x \; \mbox{is a  point at which}\; u_0~\mbox{attains a local}~ \\
\mbox{maximum}\}. \; 
\]
{We now complete the proof of Theorem~\ref{Thm3}, as the consequence of 
 Lemma~\ref{lem_u_negative} and the following lemma.}








\begin{lemma} \label{lem_no_min}
\noindent
(a)  $v_0'<0$ for all $x$;\\
\noindent
(b) $u_0$ has a unique positive maximum at $\zeta_M$, 
$u_0'>0$ on $(-\infty,\zeta_M)$ and $u_0'<0$ on $(\zeta_M,\zeta_{\mu})$;\\ 
(c) ${v_0(\zeta_M)} \to 0$ as $d \to 0$ and $u_0'<0$ on $(\zeta_M,\zeta_m)$ if $d$ is sufficiently small. 
\end{lemma}
\begin{proof} 
\noindent
(a) It is already known that $v_0'<0$ on $(-\infty,\zeta_{\mu}]$. From the definition of $\zeta_M$, we know that $u_0' \leq 0$ on $[\zeta_{M},\zeta_m]$. 
Thus 
$c_0^2 v_0'''+c_0^2 v_0''-\gamma v_0'=-u_0' \geq 0$ on $[\zeta_{\mu},\zeta_m]$ with $v_0'(\zeta_{\mu})<0$.
Also, Lemma~\ref{lem_u_negative} gives
 $v_0'(\zeta_m)<0$. 
 Then the maximum principle implies 
   $v_0'<0$ on $[\zeta_{\mu},\zeta_m]$. Combining with Lemma~\ref{lem_u_negative} yields 
    $v_0'<0$ on { \bf R}. \\

\noindent
(b) Observe that 
\begin{equation} \label{u0_prime}
d c_0^2 u_0'''+ dc_0^2 u_0''+f'(u_0) u_0'=v_0'<0 \quad \mbox{on} \; (-\infty,\infty).
\end{equation}
On $(-\infty,\zeta_M]$, $u_0 \geq \mu_3$ gives $f'(u_0)<0$. With $u_0'(\zeta_M)=0$ and $u_0'>0$ for $x$ near $-\infty$, 
the maximum principle implies 
$u_0'>0$ on $(-\infty,\zeta_M)$.  

It is already known that $u_0' \leq 0$ on $(\zeta_M,\zeta_m)$. We now 
improve the result. 
Let $\hat{x}_{\beta} \in (\zeta_M,\zeta_0)$ such that 
 $u_0(\hat{x}_{\beta})=\beta$. 
  As $dc_0^2 u_0''+dc_0^2 u_0'=v_0-f(u_0)>0$
on $[\hat{x}_{\beta},\zeta_0]$, the Hopf lemma gives $u_0'<0$ on $[\hat{x}_{\beta},\zeta_0)$. On the interval $[\zeta_M,\zeta_{\mu}]$, 
 $u_0 \geq \mu_3$ implies $f'(u_0)<0$ and 
thus $dc_0^2 u_0'''+dc_0^2 u_0''=v_0'-f'(u_0) u_0' <0$.
Since $u_0'(\zeta_M)=0$ and $u_0'(\zeta_{\mu}) < 0$, the Hopf lemma requires that $u_0'<0$ on $(\zeta_M,\zeta_{\mu})$.

\noindent
(c) It is clear from (\ref{mainc}a) that $v_0(\zeta_M) \leq f(u_0(\zeta_M))$. 
  Hence $v_0(\zeta_M) \to 0$ follows from $f(u_0(\zeta_M)) \to f(1)=0$ as $d \to 0$. 
Since ${\cal H}'_{1 \to 0}<0$ and $u_0(\cdot+\zeta_\beta) \to {{\cal H}_{1 \to 0}}
$ in $C^{\infty}_{loc}({\bf R})$ as $d \to 0$, it follows that
 $u_0'<0$ on $[\zeta_{\mu},\hat{x}_{\beta}]$ if $d$ is small.
\end{proof}

{As to distinguish different traveling wave solutions treated in the paper,
from now on the traveling front solution $(u_0,v_0)$, as stated in Theorem~\ref{mainThm}, 
is designated by
 $(u_f,v_f)$ and its speed $c_0$ is denoted by $c_f$. Recall that $c_f \equiv \max {\cal S}_f$, where ${\cal S}_f=\{c: c \in (0,\infty) 
~and~{\cal J}_f(c)=0\}$; 
 $(c_f,u_f,v_f)$ is referred to as a fast speed traveling front.}



\section{Traveling pulse solution} \label{sec_pulse}
\setcounter{equation}{0}


We now turn to the existence 
 of a traveling pulse solution. 
 Such a solution will be extracted from the admissible set ${\cal A}_p$ defined in section \ref{sec_variation}. 
Set ${\cal J}_p(c) \equiv \inf_{w \in {\cal A}_p} J_c(w)$,
{$c_p \equiv \max {\cal S}_p$, where ${\cal S}_p=\{c: c \in (0,\infty) 
\; \mbox{and} \; {\cal J}_p(c)=0\}$.} We seek a minimizer $u_p$ in ${\cal A}_p$ with the speed $c_p$. 
 Assuming ($\gamma 1$) and ($H1$) throughout the section, in the meanwhile we exhibit 
 the coexistence of a traveling pulse solution and a traveling front solution.

 
{Lemma~\ref{lem_largec} 
shows that $c_f \leq \bar{c}$ and 
the same argument yields $c_p \leq \bar{c}$. 
For the existence of a minimizer $u_p$, we remark that by setting $v_p={\cal L}_{c_p} u_p$ 
 all the lemmas in section~\ref{sec_min} and section~\ref{sec_corner} are valid or just need slight modification; for instance, Lemma \ref{lem_Bmin} is modified  
 as follows.} 

\begin{lemma} \label{lem_Bmin_p}
A minimizer $u_p$,satisfies (\ref{mainc}a) at those points where
$u_p \ne 0$ 
and $u_p \ne \beta_2$, while $v_p 
 \in C^2({\bf R})$
and it satisfies (\ref{mainc}b) everywhere.
\end{lemma}

{Also, we may give analogous notation so that $u_p$ enjoys all the properties posed in those lemmas of section~\ref{sec_linearization} and 
section~\ref{sec_positiveV}, as well as Lemmas~\ref{lem_removeBeta2} to \ref{lem_slowDecay} and Lemma~\ref{lem_zeroInterval}. 
To obtain multiple traveling wave solutions, we use the following lemma to distinguish $u_p$ from $u_f$.}

 \begin{lemma} \label{lem_pulse<0}
There exists a $d_* {\in (0,d_f]}$ such that if $d < d_*$ then $J_{c_f}(w_p) < 0$ for some $w_p \in {\cal A}_p$.
 \end{lemma}



The proof of Lemma~\ref{lem_pulse<0} requires more detailed {qualitative} behavior of $(c_f,u_f,v_f)$
{as $d \to 0$.} 
For small $d$ 
the next lemma indicates that, if $x<\zeta_M$,  the trajectory of 
$(u_f(x),v_f(x))$ varies slowly and {almost moves along the curve} $v=f(u)$
{in the $(u,v)$ plane.} 

\begin{lemma} \label{lem_slowVary}
Let $\nu \in (\mu_3,1)$ and $x_0 \in (-\infty,\zeta_M)$ be the unique point such that
$u_f(x_0)=\nu$. 
 Then 
$u_f(\cdot+x_0) \to \nu$ and $v_f(\cdot+x_0) \to f(\nu)$ in $C^{\infty}_{loc}({\bf R})$ as $d \to 0$.
\end{lemma}
\begin{proof}
Since $u_f(\cdot+x_0)$ and $v_f(\cdot+x_0)$ are uniformly bounded in high norms, along a sequence $d \to 0$, 
\begin{equation} \label{UV1}
{dc_f^2 \to \delta_0,} \quad
u_f(\cdot+x_0) \to U_0\;\; \mbox{and} \;\; v_f(\cdot+x_0) \to V_0 \;\; \mbox{in} \; C^{\infty}_{loc}({\bf R})
\end{equation}
for some 
 $U_0, V_0 \in C^{\infty}({\bf R})$. Moreover $U_0(0)=\nu$, $U'_0(0) \geq 0$ and  
\begin{equation} \label{UV2} 
\left\{ \begin{array}{rl}
\delta_0 U''_0 +\delta_0 U' _0+ f(U_0)-V_0 =& 0, \;  \\ \\
V''_0 + V'_0 =& 0. \;
\end{array} \right.
\end{equation}
 Since $0 \leq v_f \leq {\mu_3/\gamma}$,
 it follows that $0 \leq V_0 \leq {\mu_3/\gamma}$ on $[-L,L]$ for any $L>0$. 
Then $0 \leq V_0 \leq {\mu_3/\gamma}$ on $(-\infty,\infty)$ and, as a bounded 
 solution of (\ref{UV2}b), 
$V_0$ must be a constant. Thus (\ref{UV2}a) is an 
 autonomous equation; however 
the constant $V_0$ is yet to be determined. 
 
With the presence of a damping term in (\ref{UV2}a), there is no homoclinic orbit nor periodic solution to this equation.
Since $u_0$ is uniformly bounded,  the Poincare Bendixson
theorem implies that $U_0$ can only be an equilibrium solution or a heteroclinic orbit joining two equilibria. Such equilibria are the roots of $f-V_0$, 
denoted by $\rho_i$, $i=1,2,3$.  
{In the phase plane, $\rho_1$ and $\rho_3$ are saddle points while $\rho_2$ is an asymptotically stable sink or
  spiral.}  Moreover $\rho_1 \leq 0<\beta \leq \rho_2<\mu_3 \leq \rho_3 \leq 1$, 
 because $0 \leq V_0 \leq {\mu_3/\gamma}$. This indicates that $U_0$ cannot be a heteroclinic orbit with $\lim_{x\to -\infty}U_0(x) = \rho_1$, since $\mu_3 \leq U_0 \leq 1$. Observe that $U_0$ is 
non-decreasing
 on the interval $(-\infty,0]$. Then $\lim_{x\to -\infty}U_0(x) = \rho_3$ is also {impossible}, for otherwise $U_0$ has to be decreasing 
 {somewhere on $(-\infty,0]$}. 

We claim {there is no heteroclinic orbit with $\lim_{x\to -\infty}U_0(x) = \rho_2$. 
Let $Q_0(x)= - F(U_0(x))  - V_0U_0(x)$ and
\begin{equation} \label{UV3}
Q(x)= \frac{\delta_0}{2}{U'_0(x)}^2 - F(U_0(x))  - V_0U_0(x).
\end{equation}
Direct calculation gives
\begin{equation} \label{UV4}
\frac{dQ}{dx} 
 = -\delta_0 {U'_0}^2
\end{equation}
and thus}
\begin{equation} \label{Q}
\lim_{x\to -\infty}{Q_0(x) >  \lim_{x\to \infty}Q_0(x).} 
\end{equation}
Since
 $Q_0(\rho_2) < Q_0(\rho_3)$ and $Q_0(\rho_2) < Q_0(\rho_1)$, 
we {justify the above claim and} conclude that $U_0$ is an equilibrium solution and $(U_0,V_0)=(\nu,f(\nu))$.
Hence $u_f(\cdot+x_0) \to \nu$ and $v_f(\cdot+x_0) \to f(\nu)$ in $C^{\infty}_{loc}({\bf R})$ and this is true along any sequence $d \to 0$.
\end{proof}


Let us recall from introduction that $\hat{\rho}$ is the unique point where $f$ attains its local maximum. If $ \nu \in (\hat{\rho},\mu_3^*)$, 
 the horizontal line $v=f(\nu)$ intersects the graph of $v=f(u)$ at three points; 
  in an increasing order, 
they are denoted by $\rho_1,\rho_2,\nu$. It is easy to check 
   that $\rho_1<0< \mu_2^*<\rho_2<\hat{\rho}<\nu<\mu_3^*<1$. 
  {As for convenience to distinguish the notation, such three intersection points are named 
  by $\rho_1^*, \rho_2^*$ and $\mu_3^*$ when $\nu=\mu_3^*$.}
   
Set $G(\xi,\nu) \equiv F(\xi)+f(\nu) \xi$ for $\xi \in {\bf R}$. 
For a fixed $\nu$, 
$G$ is a fouth order
 polynomial  of $\xi$, which has two local minima at $\rho_1$, $\nu$ and one
 {maximum} at $\rho_2$. 
Since $\int_{\rho_1^*}^{\mu_3^*} (f(\mu_2^*)-f(\xi)) \, d\xi=0$, if 
 {$\hat{\rho}<\mu_3< \mu_3^*$} and
 $\nu \in [\mu_3,{\mu_3^*]}$ then $f(\nu)>{f(\mu_2^*)}$ and $\int_{\rho_1}^{\nu} (f(\nu)-f(\xi))\,d\xi >0$. Clearly
 \begin{equation} \label{G_shape}
 G(\rho_1,\nu)<G(\nu,\nu)<G(\rho_2,\nu) \;,
 \end{equation}
since 
 \[
 G(\nu,\nu)-G(\rho_1,\nu)=
 \int_{\rho_1}^{\nu} (f(\nu)-f(\xi))\,d\xi \;.
 \]
For {$\nu \in [\mu_3,\mu_3^*]$,} {let $E(\nu
) \equiv \{ w \in H^1_{ex}(-\infty,0): w(0)=\nu
 \}$} and 
  \begin{equation} \label{Knu}
 K_\nu(w)\equiv \int_{-\infty}^0 e^x \{ \frac{\delta_0}{2} w'^2 + G(w,\nu) \} \, dx
 = \int_{-\infty}^0 e^x \{ \frac{\delta_0}{2} w'^2 + F(w) +f(\nu) w \} \, dx. \;
  \end{equation}
It is easy to check that $K_\nu$ is bounded from below on $E(\nu)$. \\ 

{Consider the  problem 
\begin{equation} \label{el_GW}
\delta_0 w''+ \delta_0 w' +f(w)-f(\nu)=0, \; \quad\mbox{ } \; w(0)=\nu \;. 
 \end{equation}
Though $w \equiv \nu$ is a constant solution of (\ref{el_GW}), 
the next lemma shows that it 
 is not a 
minimizer of $K_\nu$ 
over $E(\nu)$.} 
 \begin{lemma} \label{lem_GW}
There is a unique minimizer $w_\nu$ 
 of $K_\nu$ 
over $E(\nu)$ and $K_\nu(w_\nu)<K_\nu(\nu)$. {$w'_\nu>0$ on $(-\infty,0]$} and 
 $w_\nu \to \rho_1$ as $x \to -\infty$. 
 \end{lemma}
 \begin{proof}
As in Lemma~\ref{lem_w0_nagumo}, the existence of a minimizer $w_\nu$ of $K_\nu$ over $E(\nu)$ can be obtained by 
 variational argument. A standard cut off technique enable us to show that 
 $\rho_1 \leq w_\nu \leq \nu$. 
If $w \equiv \nu$, a simple calculation gives 
 \begin{equation} \label{hatW}
 K_\nu(\nu)=G(\nu,\nu)=F(\nu)+\nu f(\nu) \;.
 \end{equation}

Recall that $\rho_1^*<\rho_2^*<\mu_3^*$ are the three intersection points of the
horizontal line $v=f(\mu_3^*)$ and the curve $v=f(u)$.
As in (\ref{Knu}), let 
  \[
 K_{\mu_3^*}(w)\equiv \int_{-\infty}^0 e^x \{ \frac{\delta_0}{2} w'^2 + G(w,\mu_3^*) \} \, dx
 \;.
 \]
Similar to ${\cal H}_{1 \to 0}$, denoted by ${\cal H}_{\rho_1^* \to \mu_3^*}$ the unique heteroclinic orbit of 
\begin{equation} \label{el_U_travel}
\delta_0 (e^x w')'+e^x (f(w)-f(\mu_3^*))=0;
\end{equation}
  ${\cal H}_{\rho_1^* \to \mu_3^*}(x) \to \rho_1^*$ as $x \to -\infty$ and ${\cal H}_{\rho_1^* \to \mu_3^*}(x) \to \mu_3^*$ as $x \to \infty$. 
  This orbit is monotone increasing and simply written as ${\cal H}$ in the following calculation. 
Also, we may let ${\cal H}(0)=\nu$ by taking a translation {if necessary.} 

We claim 
\begin{equation}
K_\nu(w_\nu)< K_\nu({\cal H}) < K_\nu(\nu)~{\mbox{if}}~\nu \in [\mu_3,\mu_3^*).
\end{equation}
Multiplying {$\delta_0 (e^x {\cal H}')'+e^x (f({\cal H})-f(\mu_3^*))=0$} 
by ${\cal H}'$ and integrating over 
$(-\infty,0]$, we obtain 
\[
\delta_0 e^x {\cal H}'^2 \big|_{x=-\infty}^0 - \int_{-\infty}^0 \delta_0 e^x  {\cal H}' {\cal H}'' \, dx - \int_{-\infty}^0 e^x (F({\cal H})+f(\mu_3^*) {\cal H})'\,dx=0. \;
\]
Integrating by parts again yields
\begin{equation} \label{K0}
\delta_0 \frac{({\cal H}'(0))^2}{2} + \int_{-\infty}^0 e^x \{ \frac{\delta_0}{2} {\cal H}'^2 + F({\cal H}) + f(\mu_3^*) {\cal H} \}\, dx
-\{ F(\nu)+ \nu f(\mu_3^*) \}=0. \;
\end{equation}
This together with $f(\mu_3^*)-f(\nu)<0$ and $\nu-{\cal H}>0$ on $(-\infty,0)$ yields 
\begin{eqnarray*}
 && \int_{-\infty}^0 e^x \{ \frac{\delta_0}{2} {\cal H'}^2 + F({\cal H}) + f(\nu) {\cal H} \}\, dx
-\{ F(\nu))+ \nu f(\nu) \}  \\
&=&- \delta_0 \frac{({\cal H'}(0))^2}{2} + \nu (f(\mu_3^*)-f(\nu)) - \int_{-\infty}^0 e^x (f(\mu_3^*)-f(\nu)) {\cal H} \, dx \\
&=& - \delta_0 \frac{({\cal H'}(0))^2}{2} + \int_{-\infty}^0 e^x (f(\mu_3^*)-f(\nu)) (\nu-{\cal H}) \, dx \\
&<& 0;
\end{eqnarray*}
that is, 
\begin{equation} \label{K2}
K_\nu({\cal H}) < K_\nu(\nu).
\end{equation}
Since ${\cal H}$ is a solution of (\ref{el_U_travel}) but not (\ref{el_GW}), it cannot be
a minimizer of 
$K_\nu$ over $E(\nu)$. This implies 
\begin{equation} \label{K3}
K_\nu(w_\nu)<K_\nu({\cal H}).
\end{equation} 
Note that 
 $\rho_1 < w_\nu < \nu$, because $w_\nu$ cannot be a constant solution. With a damping term in 
  (\ref{el_GW}), 
it follows from Poincare-Bendixson theorem that $w_{\nu}$ goes to one of 
$\{ \rho_1, \rho_2,\nu \}$
as $x \to -\infty$. A simple calculation gives
\begin{eqnarray*}
\frac{d}{dx} \left( \frac{\delta_0}{2} w_\nu'^2 - F(w_\nu) - f(\nu) w_\nu \right) 
&=& -\delta_0 w_\nu'^2, 
\;
\end{eqnarray*}
which implies 
$G(\nu,\nu)
> \lim_{x \to -\infty} G(w_\nu(x),\nu)$. 
This together with (\ref{G_shape}) shows that 
 $w_\nu \to \rho_1$ as $x \to -\infty$. The uniqueness 
  follows from phase plane analysis. \\ 



Next a sliding method {(see e.g. \cite{BN1})} can be applied to (\ref{el_GW}) to show 
$w_\nu' \geq 0$ on $(-\infty,0]$.
In fact $w_\nu'>0$. 
First $w_\nu'(0)>0$ since $w_\nu \not \equiv \nu$.
Suppose there is a $x_0 \in (-\infty,0)$ such that $w_\nu'(x_0)=0$, then 
$w_\nu''(x_0)=0$. Differentiating (\ref{el_GW}) gives 
$\delta_0 w_\nu''' +\delta_0 w_\nu'' + f'(w_\nu) w_\nu'=0$ with $w_\nu' \geq 0$ on $(-\infty,0]$. {Even though $f'(w_\nu)$ is not definite in 
sign, 
we may still apply 
 Hopf lemma at $x_0$ \cite{PW} 
to conclude that $w_\nu''(x_0) \ne 0$, which yields a contradiction.}
\end{proof}
\noindent
{{\bf Proof of Lemma \ref{lem_pulse<0}.}}
{Pick a $\nu \in (\mu_3,\mu_3^*)$} and
let $\zeta_\nu<\zeta_M$ be the unique point where $u_f(\zeta_\nu)=  {\nu}$. 
Lemma \ref{lem_finite_c} indicates that 
 as $d \to 0$, $\zeta_M \to -\infty$ so that 
  $\zeta_\nu \to -\infty$. 
Recall from Lemma \ref{lem_GW}
 and define
\begin{equation} \label{overlineW}
u_{new} \equiv  \left\{ \begin{array}{ll} u_f, &  \mbox{if} \; x >\zeta_\nu, \\
  w_\nu(\cdot-\zeta_\nu), & \mbox{if} \; x \leq \zeta_\nu \;.
\end{array} \right. 
\end{equation}  
It is clear that $u_{new}$ is in the class $-/+/-$, 
since Lemma~\ref{lem_GW} asserts that   
 $w_\nu \to {\rho_1}$ as $x \to -\infty$ and ${\rho_1}< \rho_1^*<0$. 
                 
Let $U_f \equiv u_f(\cdot+\zeta_\nu)$ {and $v_{new} \equiv {\cal L}_{c_f} u_{new}$}. Then 
\begin{eqnarray*}
&& J_{c_f}(u_{new})-J_{c_f}(u_f) \nonumber \\
&=&  \int_{-\infty}^{{\zeta_{\nu}}}   \{ [\frac{\delta_0}{2} w_\nu'^2(\cdot-\zeta_\nu)
+ F(w_\nu(\cdot-\zeta_\nu)) 
 + f({\nu}) w_\nu(\cdot-\zeta_\nu)] - [ \frac{{\delta_0}}{2} u_f'^2 +F(u_f) +  f({\nu}) u_f ]   \nonumber \\
&& 
+ 
(\frac{dc_f^2-\delta_0}{2}) 
(w_\nu'^2(\cdot-\zeta_\nu)-u_f'^2)
+ \frac{1}{2} {(w_\nu(\cdot-\zeta_\nu)-u_f)} (v_f+v_{new} -2 f({\nu})) \} e^x \; dx  \nonumber \\
& \leq & 
e^{\zeta_\nu} 
\int_{-\infty}^0  e^z \{ [ \frac{\delta_0}{2} w_\nu'^2
+F(w_\nu)+f({\nu}) w_\nu
] - [ \frac{\delta_0}{2} U_f'^2+F(U_f)+f({\nu}) U_f
]  \} \; dz \nonumber \\
&& +C e^{\zeta_\nu}|dc_f^2 -\delta_0| 
 + e^{\zeta_\nu}\int_{-\infty}^0  e^z \{ \frac{1}{2} (w_\nu-U_f) (v_f(\cdot+\zeta_\nu)+v_{new}(\cdot+\zeta_\nu)-2 f({\nu}))
 \} \; dz, \nonumber \\ 
 \end{eqnarray*}
 where the  positive constant $C$ is not depending on 
  $d$, because $|u_f'|$ is uniformly bounded. 
  
 

Since $w_\nu$ is the unique minimizer 
 of $K_\nu$ 
over $E(\nu)$, 
 there exists a positive constant $\tilde M $ such that 
 \begin{eqnarray} \label{B1}
 \int_{-\infty}^0  e^z \{ [ \frac{\delta_0}{2} w_\nu'^2
+F(w_\nu)+f({\nu}) w_\nu
] - [ \frac{\delta_0}{2} U_f'^2+F(U_f)+f({\nu}) U_f
]  \} \; dz \leq -\tilde M.
\end{eqnarray}
Note that $w_\nu \to \rho_1$ as $x \to -\infty$ but $U_f \geq 0$ on $(-\infty,0]$, hence $\tilde M$ in (\ref{B1}) can be chosen independent of $d$ even though $U_f$ is depending on $d$.   
 {Clearly $|u_{new}| \leq 3
 $ and thus 
  $|v_{new} | \leq 3
  /\gamma$.} 
Take a large $L$ such that 
\[
\frac{1}{2} \int_{-\infty}^{-L} e^z |w_\nu-U_f| 
(|v_{new}(\cdot+\zeta_\nu)|+v_f(\cdot+\zeta_\nu)+2 f({\nu})) \,dz \leq \tilde M/5. 
\]
Let $\tilde v \equiv v_{new}(\cdot+\zeta_\nu)-v_f(\cdot+\zeta_\nu)$. With 
$u_{new}(x+\zeta_\nu)-u_f(x+\zeta_\nu)=0$ for $x>0$, invoking (\ref{v_bound}b) yields 
 \[
\| \tilde v' \|_{L^2_{ex}} \leq \frac{2 \sqrt 3}{c_f^2} 
\] 
This implies $\| \tilde v \|_{L^{\infty}[-L,\infty)} \to 0$, using $c_f \to \infty$ as $d \to 0$; {in other words, $v_{new}(\cdot+\zeta_\nu) \to
 v_f(\cdot+\zeta_\nu)$ {uniformly} on $[-L,\infty)$. Combining with $v_f(\cdot+\zeta_\nu) \to f({\nu})$ in $C^{\infty}_{loc}({\bf R})$ from Lemma~\ref{lem_slowVary}, we apply the Lebesgue Dominated Convergence Theorem to conclude that 
\[
\frac{1}{2} \int^{0}_{-L} e^z (w_\nu-U_f )
{(v_{new}(\cdot+\zeta_\nu)+v_f(\cdot+\zeta_\nu)-}2 f({\nu})) \,dz \to 0 
\]  
as $d \to 0$. Hence there exists $d_* > 0$ such that if $d \in(0,d_*)$ then $|dc_f^2 -\delta_0| < \tilde M/5C$ and 
\[
{ | \; \frac{1}{2} \int^{0}_{-L} e^z (w_\nu-U_f )
{(v_{new}(\cdot+\zeta_\nu)+v_f(\cdot+\zeta_\nu)-}2 f({\nu})) \,dz \; | }  \leq \tilde M/5 \ .
\]
 Letting $w_p=u_{new}$ gives
 \[
J_{c_f}(w_p) < -\frac{\tilde M}{5} e^{\zeta_\nu} + J_{c_f}(u_f) = -\frac{\tilde M}{5} e^{\zeta_\nu}, 
 \]
which completes the proof. \\


{By the same lines of reasoning, $u_p < 1$ and $u_p$ cannot be identically zero on any subinterval of $(-\infty,\infty)$. In conclusion, we obtain the following result.} 
  
 \begin{lemma} \label{lem_tw}
If $d < d_*$ there is a traveling wave solution $(u_p,v_p)$ with speed $c_p$. Moreover \\ 
(a) $\frac{\mu_3}{\gamma} > v_p > 0$ on $\bf R$; \\ 
(b) there exist $\zeta_0^+<\zeta_m^+$ such that $u_p(\zeta_0^+)=0$, $u_p <0$ on $(\zeta_0^+,\infty)$ and $u_p$ has a negative local minima at $\zeta_m^+$. \\ 
 \end{lemma}

Let 
\begin{eqnarray*} 
{\cal A}_*=\{ w \in H^1_{ex}({\bf R}):   \int_{\bf R}e^x w_x^2 \, dx=2, \ -M_1 \leq w \leq \beta_2,
 \; w \; \mbox{is in the class} \; +/-\}. 
\end{eqnarray*}

Since ${\cal A}_* \subset {\cal A}_p$, as a {crucial} step
 to confirm that $(u_p,v_p)$ is a traveling pulse, we need to 
eliminate the possibility that $u_{p} \in {\cal A}_*$; {this will allow us to conclude that  $u_p$ changes sign
twice.} 
Let $\zeta_\beta^* = \sup \{ x: u_p(x)=\beta_1 \}$. To illustrate the profile of $(u_p,v_p)$, we look at 
 the asymptotical {behavior} of 
$(u_p^{(n)},v_p^{(n)})$ as $d^{(n)} \to 0$.} 

\begin{lemma} \label{lem_finite_cp}
\noindent
(a) If $d \to 0$ then $dc_p^2 \to \delta_0$, $\sup {u_p} \to1$ and $u_p(\cdot+\zeta_\beta^*) \to {{\cal H}_{1 \to 0}}$ in $C^{\infty}_{loc}({\bf R})$; \\
\noindent
(b) If $\sup {u_p} > \rho > \mu_3$ and $\zeta_\rho = \sup \{ x: u_p(x)=\rho \}$ then $\inf \{ u_p(x): x< \zeta_\rho \} \leq \mu_3$. 
\end{lemma}
 \begin{proof}
The proof of (a) is similar to that of
  Lemma~\ref{lem_finite_c}; we do not give detail. 
  
Suppose $\inf \{ u_p(x): x< \zeta_\rho \} > \mu_3$, then we may argue like Lemma~\ref{lem_no_min} to conclude that $u_p$ is monotone
near $x=-\infty$ and $\lim_{x\to -\infty}$
$u_p(x)$ exists. Then $\lim_{x\to -\infty}u_p(x) > \mu_3$ implies $\lim \inf_{x\to -\infty}v_p(x) > \frac{\mu_3}{\gamma}$. This is incompatible with Lemma~\ref{lem_tw}, and thus completes the proof. 
 \end{proof}
 
Let $\zeta_{M}^* \equiv \sup \{x: u_p~\mbox{attains a local}~\mbox{maximum at }x\}$. Similar to the argument for studying the profile of $u_f$, we conclude that $u_p$ is decreasing on $(\zeta_{M}^*,\zeta_m^+)$.

\begin{lemma} \label{em_slowVary2}
{Assume that $u_p \in {\cal A}_*$.}
Let $\nu \in (\mu_3,1)$ and $\zeta_\nu^* = \sup \{ x: u_p(x)=\nu~and~x \in (-\infty,\zeta_{M}^*) \}$. 
 Then 
$u_p(\cdot+\zeta_\nu^*) \to \nu$ and $v_p(\cdot+\zeta_\nu^*) \to f(\nu)$ in $C^{\infty}_{loc}({\bf R})$ as $d \to 0$.
\end{lemma}
\begin{proof}
The proof is analogous to that of
  Lemma~\ref{lem_slowVary}. 
  The same argument leads to 
  \begin{equation} \label{UVp}
{dc_f^2 \to \delta_0,} \quad
u_p(\cdot+\zeta_\nu^*) \to U_0\;\; \mbox{and} \;\; v_p(\cdot+\zeta_\nu^*) \to V_0 \;\; \mbox{in} \; C^{\infty}_{loc}({\bf R}),
\end{equation}
and $(U_0,V_0)$ satisfies (\ref{UV2}). 
We claim that $(U_0,V_0)$ cannot be a 
heteroclinic orbit: As in Lemma~\ref{lem_slowVary}, $\lim_{x\to -\infty}U_0(x) = \rho_2$ is impossible. 
Also, $\lim_{x\to \infty}U_0(x) = \rho_2$ cannot happen, in view of 
 the definition of $\zeta_\nu^*$. The case $\lim_{x\to -\infty}U_0(x) = 0$ enforces that $V_0=0$, but there is no such a heteroclinic orbit. The proof is complete.
  \end{proof}
\noindent  
{\bf Proof of Theorem~\ref{Thm2}.}
We argue indirectly to eliminate the possibility that $u_{p} \in {\cal A}_*$; indeed with Lemmas~\ref{lem_finite_cp} and \ref{em_slowVary2} at hand, then a modified version of the proof of 
 Lemma~\ref{lem_pulse<0} leads to a contradiction. Hence $u_p$ has to change sign twice.
 
 Set $\zeta_{0}^{-} \equiv \sup \{ x< \zeta_{M}^*: \; u_p(x)=0 \}$.  
Since $c_p^2 v_p''+c_p^2 v_p'=\gamma v_p -u_p>0$ on $(-\infty,\zeta_0^-]$, $v_p$ cannot
 have a local maximum in this interval. 
Hence $\lim_{x \to -\infty} v_p(x)$ exists. 
Duplicating the proof of Lemma~\ref{lem_uLimit}  
yields $(u_p,v_p) \to (r_1, s_1)$ as $x \to -\infty$. Clearly $r_1 \leq 0$. Thus the equilibrium $(r_1, s_1)$ 
has to be $(0,0)$, which completes the proof. \\ 

\section{Proof of 
Theorem~\ref{opp}} \label{oppf}
\setcounter{equation}{0}

{Recall that the energy level of a constant steady state $(u,v)=(\mu, \mu/\gamma)$ is 
 $L_{\gamma}(\mu,0) = \int_0^{\mu} (\xi/\gamma -f (\xi)) \, d\xi$.
  When $\gamma > \gamma_*$, 
   simple calculation yields $L_{\gamma}(\mu_2,0)  > L_{\gamma}(0,0) > L_{\gamma}(\mu_3,0)$. 
We may adapt the proof of Theorem~\ref{mainThm} to obtain a 
 traveling front solution $(c,u,v)$ of (\ref{FN})  
such that $c > 0$, $\lim_{x \to \infty}(u,v) = (0,0)$ 
  and $\lim_{x \to -\infty}(u,v) = (\mu_3,\mu_3/\gamma)$; in this case 
the argument for showing ${\cal J}_f(\underline{c})<0$ becomes easier since the invader has a lower energy. We note that 
no further restriction on $\gamma$ or $\beta$ for carrying out the truncation argument, except that $d$ satisfies ($H2$) instead of ($H1$). } 

{Next an insightful observation on making change of variables enables us to complete the proof. 
Substituting $U=\mu_3-u$ and $V=\mu_3/\gamma-v$ into (\ref{mainc}) gives 
\begin{eqnarray} \label{mainc1}
\left\{  \begin{array}{rl}
\displaystyle dc^2 U_{xx}+dc^2 U_x+\tilde{f}(U)-V &= 0, \\ \\
\displaystyle c^2 V_{xx}+c^2 V_x+U-\gamma V &=0. 
\end{array} \right.
\end{eqnarray}
Clearly $\tilde{f}$ is a cubic polynomial and the rest can be easily checked. }

\section{Appendix
} \label{App}
\setcounter{equation}{0} 
Recalling from \eqref{theta1}, we now check the conditions to ensure \eqref{t2a}. 
Let $M_\gamma$ be the unique positive number such that $f(-M_\gamma)=1/\gamma$.   
{Then 
\begin{equation} \label{t2}
f(\mu_3)+f'(\mu_3) (\xi-\mu_3)> f(\xi) \quad \mbox{on} \; [-M_\gamma,\mu_3) \cup (\mu_3,1]
\end{equation}
  if and only if \eqref{t2} 
   is satisfied at $\xi=-M_\gamma$. Thus it suffices to check
\begin{equation} \label{t5}
f(\mu_3)-f'(\mu_3) (M_\gamma+\mu_3)> \frac{1}{\gamma} \;.
\end{equation}
When $\gamma=\gamma_*$, \eqref{t5} is reduced to verifying
\[
\beta (M_\gamma+ \frac{2}{3} (1+\beta) )> \frac{1-2\beta}{3\gamma_*} = \frac{(1-2\beta)^2 (2-\beta)}{27} \;.
\]
This is clearly satisfied when $\beta$ is close to $1/2$, but not when $\beta$ is close to $0$. For instance, if $\beta_0 \in (0,1/2)$ and 
\begin{equation} \label{beta0}
\frac{2 \beta_0 (1+\beta_0)}{3} = \frac{(1-2\beta_0)^2 (2-\beta_0)}{27} \;, 
\end{equation}
then for $\beta \in (\beta_0,1/2)$, there exists  
 smallest number $\tilde{\gamma_2} 
 \in (\tilde{\gamma_1}, \gamma_*)$ 
such that \eqref{t2} holds 
 whenever $\gamma \in (\tilde{\gamma_2}, \gamma_*)$.
Take $\theta_2 > 0$ such that 
\[
f(-M_\gamma-\theta_2)=\frac{1+\theta_1}{\gamma}.
\]
Recall that $\beta_2 = 1+\theta_1$ and set $M_1=M_\gamma+\theta_2$. 
 It is easily checked that
$f(\xi) \geq \beta_2/\gamma$ for all $\xi \leq -M_1$. 
Hence if we pick 
$\theta_1$ (and therefore $\theta_2$) sufficiently small,
 by continuity \eqref{t2a} holds
 whenever $\gamma \in (\tilde{\gamma_2},\gamma_*)$.}

\vspace{.5in}
\noindent
{\bf \large Acknowledgments}
Research is supported in part  by the Ministry of Science and Technology, Taiwan, ROC. Part of the work was done when Chen was visiting the University of Connecticut and Choi was visiting the National Center for Theoretical Sciences, Taiwan, ROC and Shandong University, PRC.

\bibliographystyle{plain}

\end{document}